\documentclass[acmsmall, natbib, authorversion]{acmart-mod}
\pdfoutput=1
\usepackage[normalem]{ulem}
\usepackage{color}
\usepackage{mathrsfs}
\usepackage{graphicx}\graphicspath{{figures/}}
\newcommand{\curry}{\bar}
\newcommand{\Lideal}{\mathcal{E}}

\newcommand{\Borel}{\mathcal{B}}

\newcommand{\floor}[1]{\lfloor #1 \rfloor}

\makeatletter
\def\imod#1{\allowbreak\mkern10mu({\operator@font mod}\,\,#1)}
\makeatother

\newcommand{\zj}{{\emptyset'}}

\newcommand{\ProkBall}[1]{\ensuremath{P(#1)}}
\newcommand{\ProkBallS}[2]{\ensuremath{P_{#1}(#2)}}

\newcommand{\Measures}{\mathcal{M}}
\newcommand{\ProbMeasures}{\Measures_1}

\newcommand{\cA}[1][{}]{\mathcal{A}(#1)}
\newcommand{\cB}{\mathscr{B}}
\newcommand{\cD}{\mathcal{D}}

\newcommand{\defn}[1]{{\bf{#1}}}

\newcommand{\Ohm}{{\mathrm \Omega}}

\DeclareMathOperator{\supp}{supp}

\renewcommand{\L}{\\L}

\newcommand{\Naturals}{\mathbb{N}}
\newcommand{\Nats}{\Naturals}

\newcommand{\rv}[1]{\mathsf{#1}}

\DeclareMathOperator{\Bernoulli}{Bernoulli}

\renewcommand{\Pr}{\mathbf{P}}
\renewcommand{\Pr}{\mathbb{P}}
\newcommand{\defas}{:=}

\newcommand{\dee}{\mathrm{d}}
\newcommand{\given}{\mid}
\newcommand{\st}{\,:\,}

\newcommand{\Ind}{\mathbf 1}

\def\[#1\]{\begin{align}#1\end{align}}

\newcommand{\BasicEvents}{\mathscr{F}}

\newcommand{\PrCondProbFuncEval}[4]{{#3[#1| #2 = #4]}}
\newcommand{\PrCondProbFunc}[3]{\PrCondProbFuncEval {#1}{#2}{#3}{\cdot\;}}

\newcommand{\CondProbFuncEval}[3]{\PrCondProbFuncEval{#1}{#2}{\Pr}{#3}}
\newcommand{\CondProbFunc}[2]{\PrCondProbFunc{#1}{#2}{\Pr}}

\newcommand{\PrCondProb}[3]{{#3[#1 | #2]}}
\newcommand{\CondProb}[2]{\PrCondProb{#1}{#2}{\Pr}}

\newcommand{\EventInd}[1]{\Ind_{\{\rv X \in A\}}}
\newcommand{\RND}[2]{\frac {\dee #1}{\dee #2}}
\newcommand{\theset}[1]{\{#1\}}
\newcommand{\as}{\textrm{a.s.}}

\newcommand{\<}{\ensuremath{\langle}}
\renewcommand{\>}{\ensuremath{\rangle}}

\newcommand{\pars}{\,\cdot\;}

\newcommand{\spars}{\,\cdot\,}

\newcommand{\Reals}{\ensuremath{\mathbb{R}}}
\newcommand{\Rationals}{\ensuremath{\mathbb{Q}}}
\newcommand{\Cantor}{\{0,1\}^\infty}

\newcommand{\NNReals}{\Reals_{+}}

\title{On the Computability of Conditional Probability}

\author{Nathanael L.\ Ackerman}
\affiliation{%
  \institution{Harvard University}
  \streetaddress{Department of Mathematics, One Oxford St.,}
  \city{Cambridge}
  \state{MA}
  \postcode{02138--2901}}
\email{nate@math.harvard.edu}

\author{Cameron E.\ Freer}
\affiliation{%
  \institution{Massachusetts Institute of Technology}
  \streetaddress{Department of Brain and Cognitive Sciences, 77 Massachusetts Ave.,}
  \city{Cambridge}
  \state{MA}
  \postcode{02139--4301}}
\email{freer@mit.edu}

\author{Daniel M.\ Roy}
\affiliation{%
  \institution{University of Toronto}
  \streetaddress{Department of Statistical Sciences, 100 St.~George St.,}
  \city{Toronto}
  \state{ON}
  \postcode{M5S 3G3 and Vector Institute, MaRS Centre, West Tower, 661 University Ave., Suite 710, Toronto, ON M5G 1M1}}
\email{droy@utstat.toronto.edu}

\addtocontents{toc}{\setcounter{tocdepth}{-1}}
\begin{abstract}
As inductive inference and machine learning methods in computer science see continued success, researchers are aiming to describe ever more complex probabilistic models and inference algorithms. 
It is natural to ask whether there is a universal computational procedure for probabilistic inference.
We investigate the computability of conditional probability, a fundamental notion in probability theory and a cornerstone of Bayesian statistics.
We show that there are computable joint distributions with noncomputable conditional distributions, ruling out the prospect of general inference algorithms, even inefficient ones.  Specifically, we construct a pair of computable random variables in the unit interval such that the conditional distribution of the first variable given the second encodes the halting problem.  Nevertheless, probabilistic inference is possible in many common modeling settings, and we prove several results giving broadly applicable conditions under which conditional distributions are computable.  In particular, conditional distributions become computable when measurements are corrupted by independent computable noise with a sufficiently smooth bounded density.
\end{abstract}

\begin{document}

\theoremstyle{acmdefinition}
\newtheorem{remark}[theorem]{Remark}

\maketitle
\tableofcontents
\addtocontents{toc}{\setcounter{tocdepth}{2}}

%%%%%%%%%%%%%%%%%%%%%%%%%%%%%%%%%%%%%%%%%%%%%%
%%%%%%%%%%%%%%%%%%%%%%%%%%%%%%%%%%%%%%%%%%%%%%

%\newpage

\section{Introduction}

The use of probability to reason about uncertainty 
has
wide-ranging applications
in science and engineering,
and some of the most important computational problems 
relate to
\emph{conditioning},
which is used to perform
Bayesian inductive reasoning
in probabilistic models.
As researchers have faced more complex phenomena, their representations have also increased in complexity,
which in turn has led to more complicated inference algorithms.
It is natural to ask whether there is a universal inference algorithm --- in other words,
whether it is possible
to automate probabilistic reasoning via a
general 
procedure
that can compute conditional probabilities for an \emph{arbitrary}
computable joint distribution.

We demonstrate that
there are computable joint distributions with noncomputable
conditional distributions.
As a consequence, no general algorithm for
computing conditional probabilities can exist.
Of course, the fact that generic algorithms cannot exist for computing conditional probabilities does not rule out the possibility that large classes of distributions may be amenable to automated inference.
The challenge for mathematical theory is to explain the widespread
success of probabilistic methods and characterize
the circumstances when conditioning is possible.
In this vein, we describe broadly
applicable conditions under which conditional probabilities are computable.

We begin by describing a setting, \emph{probabilistic programming},
that motivates the search for these results.
We proceed to describe the technical frameworks for our results,
\emph{computable probability theory} and the modern formulation of
\emph{conditional probability}. We then highlight related work, 
and end the introduction with a summary of results of the paper.

\subsection{Probabilistic Programming}

Within probabilistic artificial intelligence and machine learning,
\emph{probabilistic programming} 
provides
formal languages and algorithms for describing and computing answers from probabilistic models.
Probabilistic programming languages themselves
build on modern programming languages and their facilities for
recursion, abstraction, modularity, etc.,
to enable practitioners to define intricate,
in some cases infinite-dimensional, models by implementing a generative
process that produces an exact sample from the model's joint distribution.
Probabilistic programming languages have been the focus of a long tradition of research within programming languages, model checking, and formal methods. For some of the early approaches within the AI and machine learning community, see, e.g., the languages
PHA \citep{Poole}, 
IBAL \citep{Pfeffer01},
Markov Logic \citep{DBLP:journals/ml/RichardsonD06},  
$\lambda_\circ$~\citep{1452048}, Church \citep{GooManRoyBonTen2008},
\nobreak{HANSEI} \citep{DBLP:conf/dsl/KiselyovS09}, and
Infer.NET \citep{InferNET10}.

In many of these languages, one
can easily represent the higher-order stochastic processes (e.g.,
distributions on data structures, distributions on functions, and distributions on distributions) that are
essential building blocks in modern nonparametric Bayesian statistics.
In fact, the most expressive such languages
are each capable of
describing the same robust class as the others --- the class of
\emph{computable distributions},
which delineates those from which a probabilistic Turing machine can sample to arbitrary
accuracy.

Traditionally, inference algorithms for probabilistic models have
been derived and implemented by hand.  In contrast,
probabilistic programming systems have introduced varying degrees of
support for computing conditional distributions.
Given the rate of progress toward broadening the scope of these
algorithms, one might hope that there would eventually be
a generic algorithm supporting
the entire class of computable distributions.

Despite recent progress towards a general such algorithm,
support for conditioning
with respect to continuous random variables has
remained incomplete.
Our results explain why this is necessarily the case.

\subsection{Computable Probability Theory}
In order to study computable probability theory and the computability of conditioning, we work within the framework of Type-2 Theory of Effectivity (TTE) and
use appropriate representations for topological and measurable objects such as distributions, random variables, and maps between them. This framework builds upon and contains as a special case ordinary Turing computation on discrete spaces,
and gives us
a basis for precisely
describing the operations that
probabilistic programming languages
are capable of performing.

In particular, we
study the computability of distributions on 
computable Polish spaces
including, e.g.,
certain spaces of distributions on distributions.
In Section~\ref{compprob} we present the necessary definitions and results from computable probability theory.

%%%
\subsection{Conditional Probability}
For an experiment with a discrete set of outcomes, computing
conditional probabilities is, in principle,
straightforward as it is simply a ratio of probabilities.
However, in the case of conditioning on the value of a continuous random variable, this ratio is undefined.
Furthermore, in modern Bayesian statistics, and especially the probabilistic programming setting,
it is common to place distributions on
higher-order objects,
and so one is already in a
situation where elementary notions of conditional probability are insufficient
and more sophisticated measure-theoretic notions are necessary.

Kolmogorov \citeyearpar{MR0362415} gave
an axiomatic characterization of conditional probabilities and an abstract construction of them using Radon--Nikodym derivatives,
but this definition and construction do not yield a general recipe for their calculation.
There is a further problem: 
in this setting,
conditional probabilities are formalized as measurable functions that are defined
only up to measure zero sets.  Therefore, without additional assumptions, a conditional probability is
not necessarily well-defined for any particular value of the conditioning random variable.
This has long been understood as a challenge
for statistical applications,
in which one wants to evaluate conditional probabilities given particular values for observed random variables.
In this paper, we are therefore especially interested in situations where
it makes sense to ask for the conditional distribution given a particular point.  One of our main results is in the setting where 
there is a unique continuous conditional distribution.
In this case, conditioning yields a canonical answer, which is a natural desideratum for statistical applications.

A large body of work in probability and statistics 
is concerned with the derivation of conditional probabilities and distributions in special circumstances,
each situation often requiring some special insight into the structure of the answer, 
especially when it was desirable for conditional probabilities and distributions to be defined at points, as in Bayesian statistical applications.
This state of affairs motivated work on constructive definitions of conditioning (such as those due to Tjur \citeyearpar{MR0345151, MR0123456, MR595868}, Pfanzagl \citeyearpar{MR548898}, and Rao \citeyearpar{MR970965, MR2149673}), 
although this work has not been sensitive to issues of computability.

Under certain conditions, such as when conditional densities exist, conditioning can proceed using the classic Bayes' rule; however, it may not be possible to compute the density of a computable distribution (if the density exists at all),
as we describe
in Section~\ref{cds}.

We recall the basics of the measure-theoretic approach to conditional probability in Section~\ref{conddist}, and in Section~\ref{ccd2} we use
notions from computable probability theory to consider the sense in which
conditioning could be potentially computable.

%%%%%%%%%%%%%%%%%%%%%%%%%%%%%%
%%%%%%%%%%%%%%%%%%%%%%%%%%%%%%

\subsection{Other Related Work}
\label{related}

We now describe several other connections between conditional probability and computation.

\subsubsection{Complexity Theory of Finite Discrete Distributions}
Conditional probabilities for computable distributions on finite, discrete sets
are clearly computable, but may not be efficiently so.
In this finite discrete setting, there are already interesting questions of computational complexity, which
have been explored by a number of authors through extensions of Levin's theory of average-case complexity \citep{MR822205}.
For example, under cryptographic assumptions, it is difficult to sample
from the conditional
distribution of a uniformly distributed binary string of length $n$ given
its image under a one-way function.
This can be seen to follow from the work of Ben-David, Chor, Goldreich, and Luby \citeyearpar{MR1160461}
in their theory of polynomial-time samplable
distributions, which has since been extended by Yamakami
\citeyearpar{MR1736639} and others.
Other positive and negative complexity results have been obtained in the particular case of Bayesian networks
by Cooper \citeyearpar{MR1045483} and Dagum and Luby \citeyearpar{MR1457877, MR1216898}.
Extending these complexity results to the more general setting considered here
could bear on the practice of statistical AI and machine
learning.

\subsubsection{Computable Bayesian Learners}
Osherson, Stob, and Weinstein \citeyearpar{MR973115} study learning theory in the setting of \emph{identifiability in the limit} (see
\citep{Gold67} and \citep{MR0195725}
for more details on this setting) and prove that a certain type of ``computable Bayesian''
learner fails to identify the index of a (computably enumerable) set
that is ``computably identifiable''
in the limit.
More specifically,
a ``Bayesian learner'' is required to return an index for a set with the highest conditional probability
given a finite prefix of an infinite sequence of random draws from the unknown set.
An analysis by Roy \citeyearpar{roy-phd}
of their construction reveals that the conditional distribution of the
index given the infinite sequence is an everywhere discontinuous
function (on every measure one set), hence noncomputable
for much the same reason as our elementary construction
involving a mixture of measures concentrated on the rationals and on the irrationals (see Section~\ref{Sec:discontinuous}). 
As we argue,
in the context of statistical analysis,
it is more appropriate
to study the conditioning operator
when it is restricted to those
random variables whose conditional distributions admit versions that
are continuous everywhere, or at least on a measure one set.

\subsubsection{Induction with respect to Universal Priors}
Our work is
distinct
from the study
of conditional distributions
with respect to priors that are universal for partial computable
functions (as defined using Kolmogorov complexity)
by Solomonoff \citeyearpar{MR0172745}, Zvonkin and Levin \citeyearpar{MR0307889}, and Hutter \citeyearpar{MR2354220}.
The computability of conditional distributions also has a rather different character in
Takahashi's work on the algorithmic randomness of points defined using universal Martin-L\"of tests
\citep{MR2478961}.
The objects with respect to which one is conditioning in
these settings are typically not computable (e.g., the universal semimeasure is merely lower semicomputable).
In the present paper, we are interested in the problem of computing
conditional distributions of random variables that are \emph{computable},
even though the conditional distribution may itself be
noncomputable.

\subsubsection{Radon--Nikodym Derivatives}
In the abstract setting,
conditional probabilities are (suitably measurable) Radon--Nikodym derivatives.
In work motivated by questions in algorithmic randomness,
Hoyrup and Rojas~\citeyearpar{HR11}
study notions of computability for absolute continuity and for Radon--Nikodym
derivatives as elements in $L^1$, i.e., the space of integrable functions.
Hoyrup, Rojas, and Weihrauch \citeyearpar{DBLP:conf/cie/HoyrupRW11}
then show an equivalence between the problem of computing
Radon--Nikodym
derivatives as elements in $L^1$ and computing the characteristic function
of computably enumerable sets.
The noncomputability of the Radon--Nikodym derivative operator is demonstrated by 
a pair $\mu,\nu$ of
computable measures
whose Radon--Nikodym derivative $\dee \mu / \dee \nu$ is not computable as an element in $L^1(\nu)$.
However, 
the Radon--Nikodym derivatives they study
do not correspond to conditional probabilities,
and so the computability of the operator restricted to those maps arising in the construction of conditional probabilities is not addressed by this work.
The underlying notion of computability is another important difference. 
An element in $L^1(\nu)$ is an equivalence class of functions, every pair agreeing on a set of $\nu$-measure one. 
Thus one cannot, in general, evaluate these derivatives at points in a well-defined manner.
Most Bayesian statisticians would be unfamiliar and perhaps unsatisfied with this notion of computability,
especially in settings where their statistical models admit continuous versions of Radon--Nikodym derivatives that are unique, and thus well-defined pointwise.
Regardless, we will show that even in such settings,
computing conditional probabilities is not possible in general, even in the weaker $L^1$ sense.
On the other hand, our positive results do yield computable probabilities/distributions defined pointwise.

%%%%%%%%%%%%%%%%%%%%%%%%%%%%%%%%%%%%%%
%%%%%%%%%%%%%%%%%%%%%%%%%%%%%%%%%%%%%%

\subsection{Summary of Results}
\label{cond}

Following our presentation of computable probability theory and conditional probability
in Sections~\ref{compprob} through \ref{ccd2},
we provide our main positive and negative results about the computability of conditional probability, which we now summarize.
Recall that measurable functions are often defined only up to a measure-zero set; any two functions that agree almost everywhere are called \emph{versions} of each other.

In Proposition~\ref{discontinuouscond}, 
we construct random variables $\rv X$ and $\rv C$
that are computable
on a $\Pr$-measure one set,
such that every version
of the
conditional distribution map $\CondProbFunc{\rv C}{\rv X}$ 
(i.e., a probability-measure-valued function $f$ such that $f(\rv X)$ is a regular version of the conditional distribution
$\CondProb{\rv C}{\rv X}$)
is discontinuous everywhere, even when restricted to a $\Pr_{\rv X}$-measure one subset.
(We make these notions precise in Section~\ref{ccd2}.)
The construction makes use of the elementary fact that the indicator function for the rationals in the unit interval (the so-called Dirichlet function) is itself nowhere continuous.

Because every function computable on a domain $D$ is continuous on $D$,
discontinuity is a fundamental
barrier to computability, and so
this construction rules out the possibility of a completely general algorithm for conditioning.
A natural question is whether conditioning is a computable operation when
we restrict the operator to
random variables for which \emph{some} version of the conditional
distribution is continuous everywhere, or at least on a measure one set.

In fact, even under this restriction, conditioning is not even continuous, let alone computable, as we show in Section~\ref{cond-is-discon}.
We further demonstrate that 
if some computer program purports to state true facts about the conditional distribution of a computable joint distribution provided as input, 
then we can uniformly find some other
representation of the input distribution such that the program does not output any
nontrivial fact about the conditional distribution.

Our central result,
	Theorem~\ref{actualmainresult},
provides a pair of random variables that are computable on a measure one set,
but such that the conditional distribution of one variable given the other is not computable on any measure one set (though
some version 
is continuous on a measure one set).
The construction involves encoding the halting times of all Turing
machines into the conditional distribution map 
while ensuring that the joint distribution remains computable.
This result yields another proof of the noncomputability of the conditioning operation restricted to measures having conditional distributions that are continuous on a measure one set.

In 
Section~\ref{Sec:c.e. neg cont}
we extend our central result 
by constructing a pair of random
variables, again computable on a measure one set, whose conditional distribution map is noncomputable but has an everywhere continuous version with infinitely differentiable conditional probability maps.
This construction proceeds by smoothing out the distribution constructed in 
Section~\ref{Sec:c.e. neg},
but in such a way that one can still compute the halting problem relative to the conditional distribution.
This result implies that
conditioning is not a computable operation, even when we further restrict to the case where the conditional distribution
has an everywhere continuous version.

Despite the noncomputability of conditioning in general, conditional
distribution maps are often computable in practice. We provide some explanation
of this phenomenon by
characterizing several circumstances
in which conditioning \emph{is} a computable operation.
Under suitable computability hypotheses,
conditioning is computable in the discrete setting (Proposition~\ref{discreteconditioning})
and where there is a conditional density (Corollary~\ref{independentdensity}).

We also characterize a situation in which conditioning is possible in the presence of noisy data, capturing many natural models in science and engineering.
Let $\rv U$, $\rv V$, and $\rv E$ be computable random variables where $\rv U$ and $\rv E$ are real-valued,
and
suppose that $\Pr_{\rv E}$ is absolutely continuous with a bounded computable density
$p_{\rv E}$  and $\rv E$ is independent of $\rv U$
and $\rv V$.
We can think of $\rv U + \rv E$ as the corruption of an idealized measurement $\rv U$ by independent source of additive error $\rv E$.
In Corollary~\ref{indnoise}, we show that the conditional
distribution map $\CondProbFunc{(\rv U, \rv V)}{\rv U + \rv E}$ is computable (even
if $\CondProbFunc{(\rv U, \rv V)}{\rv U}$ is not).
Finally, we discuss how symmetry, in the form of \emph{exchangeability}, can contribute to the computability of
conditional distributions.

%%%%%%%%%%%%%%%%%%%%%%%%%%%%%%%%%%%%%%%%%%%%%%%%%%%%%%%%%%%
%%%%%%%%%%%%%%%%%%%%%%%%%%%%%%%%%%%%%%%%%%%%%%%%%%%%%%%%%%%

%%%%%%%%%%%%%%%%%%%%%%%%%%%%%%%%%%%%%%%%%%%%%%%%%%%%%%%%%%%
%%%%%%%%%%%%%%%%%%%%%%%%%%%%%%%%%%%%%%%%%%%%%%%%%%%%%%%%%%%

%%%%%%%%%%%%%%%%%%%%%%%%%%%%%%%%%%%%%%
%%%%%%%%%%%%%%%%%%%%%%%%%%%%%%%%%%%%%%

\section{Computable Probability Theory}
\label{compprob}

We now give some background on computable probability theory, which will enable us to formulate our results.  The foundations of the theory include notions of computability for probability measures developed by Edalat \citeyearpar{MR1461849}, Weihrauch \citeyearpar{MR1694441}, Schr\"oder \citeyearpar{MR2351942}, and G\'acs \citeyearpar{MR2159646}.
Computable probability theory itself builds off notions and results in computable analysis, specifically the Type-2 Theory of Effectivity.
For a general introduction to this approach to real computation,
see Weihrauch \citeyearpar{MR1795407}, Braverman \citeyearpar{DBLP:conf/focs/Braverman05}
or Braverman and Cook \citeyearpar{MR2208383}.

\subsection{Computable and Computable Enumerable Reals}
\label{cereal}
We first recall some elementary definitions from computability theory
(see, e.g., Rogers \citeyearpar[][Ch.~5]{MR886890}).
A set of natural numbers (potentially in some correspondence with, e.g., rationals, integers, or other finitely describable objects with an implicit enumeration)
is \emph{computable} when there is a computer program that, given $k$, outputs whether or not $k$ is in the set.
A set
is \emph{computably enumerable} (c.e.)\ when
there is a computer program that outputs every element of the set eventually.
Note that a set is computable when both it and its complement are c.e.
We say that a sequence of sets $\{B_n\}$ is computable \emph{uniformly in $n$} when there is a single computer program
that, given $n$ and $k$, outputs whether or not $k$ is in $B_n$. We say that the sequence is
c.e.\ \emph{uniformly in $n$} when there is a computer program that, on input $n$, outputs every element of $B_n$ eventually. 

We now recall basic notions of computability for real numbers (see,
e.g., \citep[][Ch.~4.2]{MR1795407} or \citep[][Ch.~1.8]{MR2548883}).
We say that a real $r$ is a \emph{c.e.\ real} (sometimes called a \emph{left-c.e.\ real})
when the set of rationals $\{ q \in \Rationals \st q < r \}$ is c.e.  
A real $r$ is \emph{computable} when both it and its negative are c.e.
Equivalently, a real is computable when there is a program that
approximates it to any given accuracy (e.g., given an integer $k$ as
input, the program reports a rational that is within $2^{-k}$ of the real).
A function $f\colon \Nats\to\Reals$ is lower semicomputable when $f(n)$ is a c.e.\ real, uniformly in $n$ (i.e., when the collection of rationals less than $f(n)$ is c.e.\ uniformly in $n$). Likewise, a function is upper semicomputable when its negative is lower semicomputable.
The function $f$ is computable if and only if it is both lower and upper semicomputable.

\subsection{Computable Polish Spaces}

Recall that a \defn{Polish space} is a topological space that admits a metric under which it is a complete separable metric space.
Computable Polish spaces, as developed in computable
analysis \citep{MR1923905,MR1222859} and effective domain theory \citep{Jens1997225,MR1600616},
provide a convenient framework for
formulating results in computable probability theory.
For consistency, we largely use definitions from
\citep{MR2519075} and \citep{MR2558734}.
Additional details about computable Polish spaces, sometimes called computable metric spaces or effective Polish spaces, can also be found
in \citep[][Ch.~8.1]{MR1795407}, \citep[][{\S}B.3]{MR2159646}, and \citep[][Ch.~3I]{MR2526093}.

\begin{definition}[Computable Polish space {\citep[][Def.~2.3.1]{MR2558734}}]
A \defn{computable Polish space} is a triple $(S,\delta,\cD)$ for which $\delta$ is a metric on the set $S$ satisfying
\begin{enumerate}
	\item $(S,\delta)$ is a complete separable metric space;
\item $\cD = \{s_i\}_{i\in\Naturals}$ is an enumeration of a dense subset of $S$, called \defn{ideal points}; and,
\item the real numbers $\delta(s_i,s_j)$ are computable, uniformly in $i$ and $j$.
\end{enumerate}
	In particular, note that condition (1) implies that the topological space determined by the metric space $(S, \delta)$ is a Polish space.

Let $B(s_i,q_j)$ denote the ball of radius $q_j$ centered at $s_i$.
We call the elements of the set
\[
	\cB_S \defas \{ B(s_i,q_j) \st s_i \in \cD \text{~and~} q_j
	\in \Rationals \text{~s.t.~} q_j > 0 \}
\]
the \defn{ideal balls of $S$}, and
fix the canonical enumeration of them induced by that of $\cD$ and $\Rationals$.
\end{definition}

Let $\Borel_S$ denote the Borel $\sigma$-algebra on a Polish space $S$, i.e.,
the $\sigma$-algebra generated by the open balls of $S$.
Let $\ProbMeasures(S)$ denote the set of Borel probability measures on $S$.

In this paper we primarily work with computable Polish spaces. As such, unless otherwise noted,
the $\sigma$-algebras will always be 
the Borel $\sigma$-algebras on such spaces --- in particular, making them standard Borel spaces.
Measurable functions between Polish spaces will always be measurable with respect to the Borel $\sigma$-algebras.
We will sometimes refer to measurable subsets of a probability space as \emph{events}.

\begin{example}
	The set $\{0,1\}$ is a computable Polish space under the discrete metric, where $\delta(0,1)=1$.

Cantor space, the set $\Cantor$  of infinite binary sequences, is a computable Polish space under its usual
metric and the dense set of eventually constant strings (under a standard enumeration of finite strings).

The set $\Reals$ of real numbers is a computable Polish space under the
Euclidean metric with the dense set $\Rationals$ of rationals (under its standard enumeration).
\end{example}

Suppose we are given a finite sequence $(T_0, \delta_0, \cD_0), \ldots, (T_{n-1},
\delta_{n-1}, \cD_{n-1})$
of computable Polish spaces.
Then the product metric space $\prod_{i = 0}^{n-1} T_i$ (with one of any of the equivalent standard product
metrics) is a computable Polish space where the
ideal points consist of all finite products of ideal points.
Furthermore, given a countably infinite such sequence 
$(T_0, \delta_0, \cD_0), (T_1, \delta_1, \cD_1), \ldots$ that is uniformly
computable and has a fixed bound on the diameter,  the product 
metric spaces
consists of the metric space
whose underlying space is $\prod_{i \in\Nats} T_i$ and
whose metric is given by $\delta(x,y) = \sum_{i\in\Nats}2^{-i}\delta_i(x, y)$;
this too can be made into a computable Polish space, by taking the ideal
points to be those sequences $(x_0, x_1,
\ldots)$ with each $x_i \in \cD_i$ such that for all but finitely many
terms $i$, the point $x_i$ is the first element in the enumeration of
$\cD_i$. Note that in both the finite and infinite case, all
projection maps are computable.

\begin{definition}[Computable point {\citep[][Def.~2.3.2]{MR2558734}}]
Let $(S,\delta,\cD)$ be a computable Polish space with $\cD = \{s_j\}_{j \in \Nats}$ and $x \in S$.  
	Given a sequence $\{i_k\}_{i \in \Nats}$ of natural numbers,
	we say that the sequence $\{s_{i_k}\}_{k \in \Nats}$ of elements of $\cD$ 
	is a \defn{representation} of the point $x$ if
	$\delta(s_{i_k},x) < 2^{-k}$ for all $k$.
	When $\{i_k\}_{i \in \Nats}$ is a computable sequence such that $\{s_{i_k}\}_{k \in \Nats}$ is a representation of $x$, we say that $\{s_{i_k}\}_{k \in \Nats}$ is a \defn{computable representation}, and that the point $x$ is \defn{computable}.
\end{definition}

\begin{remark}
A real $\alpha\in\Reals$ is computable (as in Section~\ref{cereal})
if and only if $\alpha$ is a computable point of $\Reals$ (as
a computable Polish space).
Although most of the familiar reals are computable, there are only
countably many computable reals, and so almost every real is not
computable.
\end{remark}

The notion of a c.e.\ open set (or $\Sigma^0_1$ class) is
fundamental in classical computability
theory, and admits a simple definition in an arbitrary computable
Polish space.
\begin{definition}[C.e.\ open set {\citep[][Def.~2.3.3]{MR2558734}}]
Let $(S,\delta,\cD)$ be a computable Polish space with the corresponding enumeration $\{B_i\}_{i\in\Naturals}$ of the ideal open balls $\cB_S$.
We say that $U \subseteq S$ is a \defn{c.e.\ open set} when there is some c.e.\ set $E \subseteq \Naturals$ such that $U = \bigcup_{i \in E} B_i$.
\end{definition}

Note that the class of c.e.\ open sets is closed under computable
unions and finite intersections.

A computable function can be thought of as a continuous function
whose local modulus of continuity is witnessed by a program.
It is important to consider the computability of \emph{partial}
functions, because many natural and important random variables are
continuous
only on a measure one subset of their domain.

\begin{definition}[Computable partial function {\citep[][Def.~2.3.6]{MR2558734}}]
\label{comp-partial-func}
Let $(S,\delta_S,\cD_S)$ and $(T,\delta_T,\cD_T)$ be computable Polish spaces, the latter with the corresponding enumeration $\{B_n\}_{n\in\Nats}$ of the ideal open balls $\cB_T$.
A function $f \colon S \to T$ is said to be \defn{continuous on $R \subseteq S$} when $f$ restricted to $R$ is continuous as a function from $R$, under the subspace topology to $T$.
A function $f \colon S \to T$ is said to be \defn{computable on $R \subseteq S$} when there is a computable sequence $\{U_n\}_{n\in\Naturals}$ of c.e.\ open sets $U_n \subseteq S$ such that $f^{-1}[B_n] \cap R = U_n \cap R$ for all $n \in \Naturals$. We call such a sequence $\{U_n\}_{n \in \Naturals}$ a \defn{witness} to the computability of $f$.
\end{definition}

Note that the notion of being computable on a set $R$ can be relativized to an oracle $A\subseteq \Nats$ in the obvious way. A function is continuous on $R$ if and only if it is $A$-computable on $R$ for some oracle $A$.

%%%%%%%%%%%%%%%

\begin{remark}
\label{comppointremark}
Let $S$ and $T$ be computable Polish spaces.
If $f\colon S\to T$  is computable on some subset $R\subseteq S$, then
for every \emph{computable} point $x\in R$, the
point $f(x)$ is also computable.
One can show that $f$
is computable  on $R$
when there is an oracle Turing machine that, upon being fed a
representations of points $x \in R$ on its oracle tape, computes representations of their images $f(x)\in S$. 
(For more details, see \citep[][Prop.~3.3.2]{MR2519075}.)
\end{remark}

\subsection{Notions of Computability for Functions on Probability Spaces}

The standard notion of computability of functions between computable Polish spaces is too restrictive in most cases when the 
inputs to these functions are points in a probability space. For example, the Heaviside function $f(x) = \Ind(x \ge 0)$ is not computable on any set containing a neighborhood of 0.  However, we can reliably compute the image of a Gaussian random variable under $f$, because the Gaussian random variable is nonzero with probability one, and $f$ is computable on $\Reals \setminus \{0\}$.

For a measure space $(\Ohm, \mathscr{G}, \mu)$, a set $E \in \mathscr
G$ is a \defn{$\mu$-null set} when $\mu (E) = 0$.
More generally, for $p \in [0,\infty]$, we say that $E$ is a \defn{$\mu$-measure $p$ set} when $\mu (E) = p$. 
A predicate $P$ on $\Ohm$ is said to hold
\defn{$\mu$-almost everywhere} (abbreviated \defn{$\mu$-a.e.}) if the event 
$E_P = \{ \varpi \in \Ohm \st P(\varpi) \text{ does not hold} )$ is a $\mu$-null set.  
When $E_P$ is a $\mu$-null set but $\mu$ is a probability measure,
we will instead say the event $P$ holds \defn{$\mu$-almost surely},
and we likewise say
that an event $E \in \mathscr G$ occurs \defn{$\mu$-almost surely}
(abbreviated \defn{$\mu$-a.s.}) when $\mu (E) = 1$.  In each case, we may
drop the prefix $\mu$ when it is clear from context  (in particular,
when it holds of $\Pr$). 

\begin{definition}\label{tnhethues}
	Let $S$ and $T$ be Polish spaces and $\mu$ a probability measure on $S$.
A measurable function $f\colon S \to T$ 
	is \defn{$\mu$-almost continuous} when it is continuous on a $\mu$-measure one set. 
	When $S$ and $T$ are computable Polish spaces, the measurable function $f$ is
	\defn{$\mu$-almost computable} when it is computable on a $\mu$-measure one set.  
\end{definition}

(See \citep{MR2519075} for further development of the theory of almost computable functions.)
The following result relates $\mu$-almost continuity to $\mu$-a.e. continuity, i.e., the set of continuity points being a $\mu$-measure one set.
The proofs of the following proposition and lemma are due to 
Fran\c{c}ois Dorais, Gerald Edgar, and Jason Rute
{\citeyearpar[][]{146063}}.

\begin{proposition}
\label{DER-theorem}
Let $X$ and $Y$ be Polish spaces,
let $f\colon X \to Y$ be a $\mu$-almost continuous function, and let $\mu$ be a 
probability measure on $X$.
	Then there is a $\mu$-a.e.\ continuous $g\colon X\to Y$ that agrees with $f$ $\mu$-a.e.
\end{proposition}

We will need the following technical lemma. Recall that a $G_\delta$ set is a countable intersection of open sets.
\begin{lemma}
\label{DER-lemma}
Let $X$ be a Polish space.
If $D \subseteq X$ is a nonempty $G_\delta$-set then there is a measurable map $h\colon X \to D$ such that $\lim_{x\to x_0} h(x) = x_0$ for every $x_0 \in D$.
\end{lemma}
\begin{proof}
Suppose $D = \bigcap_{n \in \Nats} U_n$, where $(U_n)_{n\in\Nats}$ is a
descending sequence of open sets such that $U_n \subseteq \bigcup_{x_0 \in
D} B(x_0,1/(n+1))$. Any measurable retraction $h\colon X \to D$ with the property that
if $x \in U_n \setminus U_{n+1}$ then $d(h(x),x) < 1/(n+1)$ will be as
required. By definition, it is always possible to find a suitable $h(x) \in
D$ for each $x \in U_0 \setminus D$. To ensure that $h$ is measurable, fix an
enumeration $(d_i)_{i \in \Nats}$ of a countable dense subset of $D$ and, if $x \in U_0 \setminus D$, define $h(x)$ to be the first element in this list that matches all the necessary requirements. (We must have $h(x) = x$ for $x \in D$ and it does not matter how $h(x)$ is defined when $x \notin U_0$ so long as the end result is measurable.)
\end{proof}

\begin{proof}[Proof of Proposition~\ref{DER-theorem}]
By a classical result of Kuratowski \citep[][I.3.B, Thm.~3.8]{MR1321597}, we may assume (after first possibly changing its value on a $\mu$-null set) that $f$ is
continuous on a $\mu$-measure one $G_\delta$ set $D \subseteq X$.

Let $h$ be as in Lemma~\ref{DER-lemma}. Then
$g = f\circ h$ is a measurable function that agrees with $f$ on $D$ and $$\lim_{x \to x_0} g(x) = f\bigl(\lim_{x \to x_0} h(x)\bigr) = f(x_0) = g(x_0)$$ for all $x_0 \in D$.
\end{proof}

\begin{remark}
	Let $S$ and $T$ be computable Polish spaces.
	A set $X\subseteq S$ is an effective $G_\delta$ set (or $\Pi^0_2$ class) when it is the
	intersection of a uniformly computable sequence of c.e.\ open sets.
Suppose that $f\colon S\to T$ is computable on $R\subseteq S$ with $\{U_n\}_{n \in \Naturals}$ a witness to the computability of $f$. One can show 
that there is an effective $G_\delta$ set 
	$R'\supseteq R$ and a function $f'\colon S \to T$ such that
$f'$ is computable on $R'$, the restriction of $f'$ to $R$ and $f$ are equal as functions, and $\{U_n\}_{n \in \Naturals}$ is a witness to the computability of $f'$. Furthermore, a $G_\delta$-code for some such $R'$ can be computed uniformly from a code for the witness $\{U_n\}_{n\in \Naturals}$.
For details, see \citep[][Thm.~1.6.2.1]{hoyrupthesis}; this generalizes a classical
result of Kuratowski \citep[][I.3.B, Thm.~3.8]{MR1321597}.
In conclusion, one can always assume that the set $R$ is an effective $G_\delta$ set. 
\end{remark}

We will introduce a weaker notion of computability for functions in Section~\ref{l1funcs}.

\subsection{Computable and Almost Computable Random Variables}
\label{comprvanddist}

Intuitively, a random variable maps an input source of
randomness to an output, inducing a distribution on the output space.
Here we will use a sequence of independent fair coin flips as our
source of randomness.
We formalize this via the probability space $(\Cantor, \BasicEvents, \Pr)$, where $\Cantor$ is
the product space of infinite binary sequences, $\BasicEvents$ is its Borel
$\sigma$-algebra (generated by the set of basic clopen cylinders extending
each finite binary sequence), and $\Pr$ is the
product measure formed from the uniform distribution on $\{0,1\}$.
Throughout the rest of the paper we will take 
$(\Cantor, \BasicEvents, \Pr)$ to be the basic probability space.
We will use a $\mathsf{SANS~SERIF}$ font for random variables.

\begin{definition}[Random variable and its distribution]
Let $S$ be a Polish space.
A \defn{random variable in $S$} is a measurable function
$\rv X \colon \Cantor \to S$.
For a measurable subset $A \subseteq
S$, we let $\{ \rv X \in A \}$ denote the inverse image 
$\rv X^{-1}[A] = \{ \varpi \in \Cantor \st \rv X(\varpi) \in A \}$, 
and for $x \in S$
we similarly define
the event $\{ \rv X = x \}$.
We will write $\Pr_{\rv X}$ for the \defn{distribution of $\rv X$}, which is the measure on $S$ defined by
$\Pr_{\rv X}(\pars) \defas \Pr\{\rv X \in \cdot \, \}$.
\end{definition}

If $S$ is a computable Polish space then we say a
random variable $\rv X$ in $S$ is a
\defn{$\Pr$-almost computable random variable} it is $\Pr$-almost computable as a measurable function.
Intuitively,
$\rv X$ is a $\Pr$-almost computable random variable when there is a program
that, given access to an oracle bit tape $\varpi \in \Cantor$, outputs a
representation of the point $\rv X(\varpi)$ (i.e., enumerates a sequence
$\{x_i\}$ in $\cD$ where $\delta(x_i,\rv X(\varpi)) < 2^{-i}$ for all $i$),
for all but a $\Pr$-measure zero subset of bit tapes $\varpi \in \Cantor$.

Even though the source of randomness is a sequence of discrete bits, there are $\Pr$-almost computable random variables with \emph{continuous} distributions, such as a uniform random variable (gotten by subdividing the unit interval according to the random bit tape) or an i.i.d.\ sequence of uniformly distributed random variables (by splitting up the given element of $\Cantor$ into countably many disjoint subsequences and dovetailing the constructions).  (For explicit constructions, see, e.g., \citep[][Ex.~3, 4]{FreerRoyAISTATS2010}.) 

It is crucial that we consider
random variables that
are merely computable on a $\Pr$-measure one subset of $\Cantor$.
To see why, consider the following example, which was communicated to us by Mart\'in Escard\'o.
For a real $\alpha\in[0,1]$,
we say that a binary random variable
$\rv X\colon \Cantor \to \{0,1\}$
is a \defn{Bernoulli}($\alpha$)
random variable when $\Pr_{\rv X}\{1\} = \alpha$.
There is a $\Bernoulli(\frac12)$ random variable
that is computable on all of $\Cantor$, given by the program that
simply outputs the first bit of the input sequence.
Likewise, when $\alpha$ is \defn{dyadic}
(i.e., a rational whose denominator is a power of 2),
there is a $\Bernoulli(\alpha)$ random variable
that is computable on all of $\Cantor$. However, this is not possible for
any other choices of $\alpha$ (e.g., $\frac13$).

\begin{lemma}\label{Bernoulliproof}
Let $\alpha\in[0,1]$ be a nondyadic real.
Every $\Bernoulli(\alpha)$ random variable $\rv X \colon \Cantor \to \{0,1\}$ is discontinuous, hence not computable on all of $\Cantor$.
\end{lemma}
%%%%%%%%%%%%%%%%%%%%%%%%%%%%%%%%%%%%%%
\begin{proof}
Assume $\rv X$ is continuous.
Let $Z_0 \defas \rv X^{-1}(0)$ and $Z_1 \defas \rv X^{-1}(1)$.
Then $\Cantor = Z_0 \cup Z_1$, and so both are
closed (as well as open).  The compactness of $\Cantor$ implies that
these closed subspaces are also compact, and so $Z_0$ and $Z_1$ can
each be written as the finite disjoint union of clopen basis
elements.  But each of these
elements has dyadic measure, hence their sum cannot be either
$\alpha$ or $1-\alpha$, contradicting the fact that
$\Pr(Z_1)= 1- \Pr(Z_0) = \alpha$.
\end{proof}
%%%%%%%%%%%%%%%%%%%%%%%%%%%%%%

On the other hand, for an arbitrary computable $\alpha\in[0,1]$, consider
the random variable $\rv X_\alpha$ given by $\rv X_\alpha(x) = 1$ if $\sum_{i=0}^\infty x_i 2^{-i-1} < \alpha$ and $0$ otherwise.  This construction, due to \citet{MR0322920}, is a
$\Bernoulli(\alpha)$ random variable and is computable on every point of $\Cantor$ other than a
binary expansion of $\alpha$.
Not only are these random variables $\Pr$-almost computable, but they can be shown to be optimal in their use of input bits,
via the classic analysis of rational-weight coins by Knuth and Yao
\citeyearpar{MR0431601}. 
Hence it is natural to focus our attention on random variables that are merely $\Pr$-almost computable.

The setting of $\Pr$-almost computable random variables is a natural one for probability theory, and the standard operations
on random variables preserve $\Pr$-almost computability, including, e.g., addition and multiplication of $\Pr$-almost computable real random variables, 
composition with $\Pr$-almost computable measurable functions, and cartesian products.

\subsection{Computable Probability Measures}
\label{sec:compdists}

We now introduce the class of computable probability measures on computable Polish spaces.

Let $(S,\delta_S,\cD_S)$ be a computable Polish space, and recall that
$\Borel_S$ denotes its Borel sets and $\ProbMeasures(S)$ its Borel
probability measures.  
Consider the subset $\cD_{P, S} \subseteq \ProbMeasures(S)$ comprised of those probability measures that are concentrated on a finite subset of $\cD_S$ and where the measure of each atom is rational,
i.e.,  $\nu \in \cD_{P, S}$ if and only if
$\nu = q_1 \mbox{\boldmath$\delta$}_{t_1} + \dotsb + q_k \mbox{\boldmath$\delta$}_{t_k}$ for some rationals $q_i \ge 0$ such that
$q_1 + \dotsb + q_k = 1$ and some points $t_i \in \cD_S$,
where for $t\in S$ the $\{0,1\}$-valued Dirac measure $\mbox{\boldmath$\delta$}_{t}$ satisfies $\mbox{\boldmath$\delta$}_{t}(A) = 1$ if and only if $t \in A$ for all measurable sets $A$.
It is a standard fact (see, e.g., G\'acs \citeyearpar[][{\S}B.6.2]{MR2159646}) that $\cD_P$ is dense in the
Prokhorov metric $\delta_P$ given by
\[
\delta_P (\mu, \nu)
&\defas \inf \left\{ \varepsilon > 0 \st \forall A \in \Borel_S,\  \mu(A) \leq \nu (A^{\varepsilon}) + \varepsilon \right \},
\]
where
\[
A^{\varepsilon} \defas \{ p \in S \st \exists q \in A, \ \delta_S(p, q) < \varepsilon \} = \textstyle \bigcup_{p \in A} B_{\varepsilon} (p)
\]
is the $\varepsilon$-neighborhood of $A$ and $B_\varepsilon(p)$ is the open ball of radius $\varepsilon$ about $p$.  Moreover, $(\ProbMeasures(S),\delta_P,\cD_{P, S})$ is a computable Polish space.  (See \citep[][Prop.~4.1.1]{MR2519075}.)  We say that $\mu \in \ProbMeasures(S)$ is a computable probability measure when $\mu$ is a computable point in $\ProbMeasures(S)$ as a computable Polish space. Note when the space $S$ is clear from context we will refer to $\cD_{P, S}$ simply as $\cD_P$.

One can define computability on the space of probability measures in other natural ways.
Early work by Weihrauch \citeyearpar{MR1694441} and M\"uller \citeyearpar{MR1694435} formalized the computability of probability measure in terms of the lower semicomputability of the measure as a function on the set of open sets and in terms of the computability of the measure as a linear operator acting on bounded continuous functions; these notions are equivalent. (See Schr\"oder \citeyearpar{MR2351942} for a more general setting.)  These notions of computability also agree with the notion of computability defined here in terms of the Prokhorov metric.

\begin{proposition}[{\citep[][Thm.~4.2.1]{MR2519075}}]
\label{lowerbounds}
Let $S$ be a computable Polish space.
A probability measure $\mu \in \ProbMeasures(S)$ is computable
if and only if
the measure $\mu(A)$ of a c.e.\ open set $A \subseteq S$ is a c.e.\ real, uniformly in $A$.
\qed
\end{proposition}
%%%%%%%%%%%%%%%%%%%
Note that the measure $\Pr$ on $\Cantor$ is a computable probability measure.

We can also characterize the class of computable probability measures in terms of the uniform computability of the integrals of bounded continuous functions:

\begin{proposition}[{\citep[][Cor.~4.3.1]{MR2519075}}]\label{intcomp}
Let $S$ be a computable Polish space, 
let $\mu$ be a probability measure on $S$,
and let $\mathcal F$ be the set of computable functions from $S$ to $\Reals^+$.
Then $\mu$ is computable if and only if $\int f\, d \mu$ is a c.e.\ real, uniformly in $f \in \mathcal F$.
\qed
\end{proposition}
\begin{corollary}
\label{intcomp2}
Let $S$ be a computable Polish space, 
let $\mu$ be a probability measure on $S$,
and let $\mathcal F$ be the set of computable functions from $S$ to $[0,1]$.
Then $\mu$ is computable if and only if $\int f\, d \mu$ is computable, uniformly in $f \in \mathcal F$.
\end{corollary}
\begin{proof}
	First observe that both $f$ and $1-f$ are non-negative functions. 
Therefore,
	by Proposition~\ref{intcomp},
	the reals $\int f\, d\mu$ and $\int (1-f)\, d \mu$ are both c.e., and hence the real $\int f\, d \mu$ is computable. 
\end{proof}

Having explained the computability of probability measures in terms of integration, we now relate it to the computability of random variables defined on computable Polish spaces.

\begin{definition}[Computable probability space {\citep[][Def.~2.4.1]{MR2558734}}]
A \defn{computable probability space} is a pair $(S,\mu)$ where $S$ is a computable Polish space and $\mu$ is a computable probability measure on $S$.
\end{definition}

%%%%%%%%%%%%%%%%%%%%%%%%

The distribution of a $\Pr$-almost computable random variable in a computable Polish space is computable.

\begin{proposition}[{\citep[][Prop.~2.4.2]{MR2558734}}]
\label{rvtodist}
Let $\rv X$ be a $\Pr$-almost computable random variable in a computable Polish space $S$.
Then its distribution
is a computable point in the computable Polish space $\ProbMeasures(S)$.
\qed
\end{proposition}

On the other hand, given a computable measure,
there is a $\Pr$-almost computable random variable with that distribution.

\begin{proposition}[{\citep[][Thm.~5.1.1]{MR2519075}}]
\label{disttorv}
Let $\mu$ be a computable probability measure on a computable Polish space $S$.
Then there is a $\Pr$-almost computable random variable in $S$ whose distribution is $\mu$.
\qed
\end{proposition}

In summary,
the computable probability measures 
on a computable Polish space 
are precisely
the distributions of $\Pr$-almost computable random variables in that space. For this result in a more general setting, see \citep[][Prop.~4.3]{MR2351942}.

Further,
if $\mu$ is a computable probability measure and
and $f$ is computable on a $\mu$-measure one set,
then the pushforward $\mu \circ f^{-1}$ is a computable distribution.
This fact, along with 
Proposition~\ref{disttorv}, shows that we have
lost no generality in taking
$(\Cantor, \BasicEvents, \Pr)$ to be our basic probability space.

All of the standard distributions (e.g., normal, uniform, geometric, exponential) found in probability textbooks, and then all the transformations of these distributions by $\Pr$-almost computable functions, are easily shown to be computable distributions.

\subsection{Weaker Notions of Computability for Functions on Probability Spaces}
\label{l1funcs}

%%%%%
Another important class of functions on a probability space is the class of $L^1$-computable functions.
For more details, including some of the history of $L^1$-computability, see Hoyrup and Rojas \citeyearpar[][{\S}3.1]{HRlayer1} and Miyabe \citeyearpar{MR3123251}.
\begin{definition}[{The metric space of $L^1(\mu)$ functions \citep[][{\S}3.1]{HRlayer1}}]
Let $\mu$ be a probability measure on a Polish space $S$, and
let $\mathcal{F}$ be the set of $\mu$-integrable functions from $S$ to $\Reals$.
Then $\delta(f, g) \defas \int|f -g|\,d\mu$ is a metric on the quotient space of
$\mathcal{F}$ defined by the equivalence relation $f\sim g$ iff $\int|f -g|\,d\mu = 0$.
This metric space is called the \defn{space of $L^1(\mu)$
functions on $S$}, and we will often speak interchangeably of a $\mu$-integrable function
$S\to\Reals$ and its equivalence class. 
\end{definition}
We will make use of the following set of $L^1$ functions.
\begin{definition}[{Ideal points for $L^1$ {\citep[][{\S}2]{MR2159646}}}]
\label{Eideal}
Let $(S, \delta, \cD)$ be a computable Polish space.
Define $\Lideal$ to be the smallest set of functions containing
the constant function $1$ and the 
functions $\{ g_{u, r, 1/n} \st u\in S,\,  r \in \Rationals,\, n \ge 1\}$,
where
\[
g_{u, r, \epsilon}(x) \defas \max\bigl (0, \, 1 - \max(0,\, \delta(x, u) - r) / \epsilon\bigr ),
\]
and closed under $\max$, $\min$, and rational linear combinations.
\end{definition}
Such functions can be thought of as continuous
analogues of step functions having a finite number of steps, each step of
which corresponds to a basic open ball with rational radius and ideal
center.
\begin{lemma}[{\citep[][Prop.~3]{HRlayer1}}]
\label{EidealDense}
Let $\mu$ be a computable probability measure on a computable Polish space
$(S,\delta, \cD)$.
The set $\Lideal$ is dense in the $L^1(\mu)$ functions on $S$,
and the distances between points in $\Lideal$ are computable under the
standard enumeration, making this space into a computable Polish space.
\qed
\end{lemma}
We say that an $L^1(\mu)$ function on a computable Polish space $S$ is
\defn{$L^1(\mu)$-computable} when it is a computable point in the
$L^1(\mu)$ functions on $S$.

\begin{lemma}[{{\citet[][Thm.~4 Claim~2 and Thm.~5 Claim~2]{HRlayer1}}}]
\label{layer}
Let $(S,\mu)$ be a computable probability space and let $T$ be a computable Polish space.
A function $f\colon S \to T$ is $L^1(\mu)$-computable if and only if
$\int f\, d \mu$ is a computable real and
for each $r \in \Nats$, the function
$f$ is computable on some set of $\Pr_{\rv X}$-measure at least $1-2^{-r}$, uniformly in $r$.
\qed
\end{lemma}

In particular, note that every integrable $\mu$-almost computable function is $L^1(\mu)$-computable.

We obtain the following immediate corollary of Lemma~\ref{layer} using the fact that if a function is $\mu$-almost computable with a computable $\mu$-integral, then we can uniformly find a collection of ideal points that converge to it in $L^1(\mu)$.

\begin{corollary}\label{L1aslimit} 
Let $(S,\mu)$ be a computable probability space and let $T$ be a computable Polish space.
Let $f_0,f_1,\dotsc \colon S \to T$ be a sequence of uniformly $\mu$-almost computable functions taking values in a computable Polish space $T$ that converge effectively in $L^1(\mu)$ to
a function $f \in L^1(\mu)$. Then $f$ is $L^1(\mu)$-computable.
\qed
\end{corollary}

\subsection{Almost Decidable Sets and Bases}

Let $(S,\mu)$ be a computable probability space.  We know that the $\mu$-measure
of a c.e.\ open set $A \subseteq S$ is a c.e.\ real. 
In general, the measure of a c.e.\ open set is not a computable real.
On the other hand, if $A$ is a decidable subset (i.e., $S \setminus A$ is c.e.\ open) then $\mu(S \setminus A)$ a c.e.\ real, and therefore, by the identity $\mu(A)+\mu(S\setminus A)=1$, we have that $\mu(A)$ is a computable real.  In connected spaces, the only decidable subsets are the empty set and the whole space. However, there exists a useful surrogate when dealing with measure spaces.

\begin{definition}[Almost decidable set {\citep[][Def.~3.1.3]{MR2558734}}]
Let $(S, \mu)$ be a computable probability space.
A measurable subset $A \subseteq S$ is said to be \defn{$\mu$-almost decidable} when there are two c.e.\ open sets $U$ and $V$ such that
$
U \subseteq A$ and $V \subseteq S \setminus A$ and $\mu(U) + \mu(V) = 1.
$
In this case we say that $(U, V)$ witnesses the $\mu$-almost
decidability of $A$.
\end{definition}

The following lemma is immediate.

\begin{lemma}[{\citep[][Prop.~3.1.1]{MR2558734}}]\label{almostdec-comp}
Let $(S,\mu)$ be a computable probability space, and let $A$ be $\mu$-almost decidable.  Then $\mu(A)$ is a computable real.
\qed
\end{lemma}

While we may not be able to compute the probability measure of ideal balls, we can compute a new basis of ideal balls for which we can.  (See also Bosserhoff~\citeyearpar[][Lem.~2.15]{MR2410914}.)

\begin{lemma}[{\cite[Thm.~3.1.2]{MR2558734}}]\label{almdecbasis}
Let $(S, \mu)$ be a computable probability space, and let $\cD_S$ be the ideal points of $S$ with standard enumeration $\{d_i\}_{i\in\Nats}$.
There is a computable sequence $\{r_j\}_{j\in\Naturals}$ of reals,
dense in the positive reals, such that the balls
$\{B(d_i,r_j)\}_{i,j\in\Naturals}$ form a basis of $\mu$-almost
decidable sets,
which we call a \defn{$\mu$-almost decidable basis}.
\qed
\end{lemma}

We now show that every c.e.\ open set of a computable probability space $(S, \mu)$ is the union of a computable sequence of $\mu$-almost decidable subsets.

\begin{lemma}[Almost decidable subsets]
\label{comppair}
Let $(S, \mu)$ be a computable probability space 
with ideal points $\{d_i\}_{i\in\Nats}$, and let $\{r_j\}_{j\in\Naturals}$ be a computable sequence of reals such that
$\{B(d_i,r_j)\}_{i,j\in\Naturals}$ is a $\mu$-almost decidable basis.
Let $V$ be a c.e.\ open set.  Then, uniformly in
$\{r_j\}_{j\in\Naturals}$ and $V$,
we can compute a sequence of $\mu$-almost decidable sets $\{V_k\}_{k\in\Naturals}$ such that
	$V_k \subseteq V_{k+1}$
	for each $k$, and
$\bigcup_{k\in\Naturals} V_k = V$.
\end{lemma}
\begin{proof}
Let $\{B_k\}_{k\in\Naturals}$ be a standard enumeration of the ideal balls of $S$ where $B_k = B(d_{m_k}, q_{l_k})$, and let
$E \subseteq \Naturals$ be a c.e.\ set such that $V = \bigcup_{k\in E} B_k$.
\ Consider the c.e.\ set
\[
F_k \defas \{ (i,j) \st \delta_S(d_i, d_{m_k}) + r_j < q_{l_k}\}.
\]
Because $\{d_i\}_{i\in \Naturals}$ is dense in $S$ and $\{r_j\}_{j \in \Naturals}$ is dense in the positive reals we have for each $k \in \Naturals$ that $B_k = \bigcup_{(i, j) \in F_k} B(d_i, r_j)$. In particular this implies that the set $F \defas \bigcup_{k \in E} F_k$ is a c.e.\ set 
	with $V = \bigcup_{(i,j) \in F} B(d_i,r_j)$. Let $\{(i_n, j_n)\}_{n \in \Naturals}$ be a computable enumeration of $F$ and let $V_k \defas \bigcup_{n\le k} B(d_{i_n},r_{j_n})$, which is $\mu$-almost decidable.  By construction, 
	$V_k \subseteq V_{k+1}$ for each $k$, 
	and $\bigcup_{k\in\Naturals} V_k = V$.
\end{proof}

Using the notion of an almost decidable set, we have the following characterization of computable measures.

\begin{corollary}\label{almostcomplementary-to-compmeas}
Let $(S,\mu)$ be a computable probability space 
with ideal points $\{d_i\}_{i\in\Nats}$, 
and let $\{r_j\}_{j\in\Naturals}$ be a computable sequence of reals such that
$\{B(d_i,r_j)\}_{i,j\in\Naturals}$ is a $\mu$-almost decidable basis.
Let $\nu \in
\ProbMeasures(S)$ be a probability measure on $S$ that is absolutely
continuous with respect to $\mu$.
  Then $\nu$ is computable uniformly in the sequence
$\{\nu(B(d_i,r_j))\}_{i,j\in\Naturals}$.
\end{corollary}
\begin{proof}
Let $V$ be a c.e.\ open set of $S$.
By Proposition~\ref{lowerbounds}, it suffices to show that $\nu(V)$ is a c.e.\ real, uniformly in $V$.
By Lemma~\ref{comppair}, we can compute a nested sequence $\{V_k\}_{k\in\Naturals}$
of $\mu$-almost decidable sets whose union is $V$. By the absolute
continuity of $\nu$ with respect to $\mu$, these sets are also
$\nu$-almost decidable.
Because $V$ is open, $\nu(V) = \sup_{k\in\Naturals} \nu(V_k)$,
which is the supremum of a sequence of reals that is computable uniformly in the
sequence $\{\nu(B(d_i,r_j))\}_{i,j\in\Naturals}$.
\end{proof}

We close with the following extension of Corollary~\ref{intcomp2}.

\begin{proposition}
\label{inttwovartoonever}
Let $S$, $T$ be computable Polish spaces, $\mu$ a probability measure on $T$, $B$ a $\mu$-almost decidable subset of $\Reals$, 
and $f\colon S \times T \to \Reals$ a bounded function, computable on $R \times T$ with $R \subseteq S$.
Then the map $s \mapsto \int_B f(s, t) \mu(\dee t)$ is a computable function, uniformly in $f$ and $B$.
\end{proposition}
\begin{proof}
This follows immediately from Propositions~3.2.3 and 4.3.1 of \citep{MR2519075}.
\end{proof}

%%%%%%%%%%%%%%%%%%%%%%%%%%%%%%%%%%%%%%

\section{Conditional Probabilities and Distributions}
\label{conddist}

Let $\mu$ be a probability measure on a measurable space of outcomes $S$,
and let $A,B \subseteq S$ be events.
Informally, given that event $A$ has occurred, the probability that event $B$ also occurs, written $\mu(B|A)$, 
must satisfy $\mu(A) \, \mu(B|A) = \mu (A \cap B)$.
Clearly $\mu(B|A)$ is uniquely defined if and only if $\mu(A) > 0$, which leads to the following definition.

\begin{definition}[Conditioning on positive-measure events]
\label{Def:conditional-prob-set}
Suppose that
$\mu(A)>0$. Then the
\defn{conditional probability of $B$ given $A$}, written $\mu(B | A)$, is defined by
\[
\mu(B | A) = \frac{\mu(B \cap A)}{\mu(A)}.
\]
\end{definition}
It is straightforward to check that, for any fixed event $A\subseteq S$ with $\mu(A)>0$,
the set function $\mu(\pars |A)$ is a probability measure.  

We will often be interested in the case where $B$ and $A$ are
events of the form  $\{ \rv Y \in D \}$ and $\{ \rv X \in C\}$.
 In this case, we define the abbreviation
\[\Pr\{\rv Y \in D \given \rv X \in C \} \defas \Pr\bigl(\{\rv Y \in D\}
\given  \{\rv X \in C \}\bigr).\]
Again, this is well-defined when $\Pr\{X \in C\} > 0$.
When $\Pr\{ \rv X = x\} >0 $, we may simply write
\[\Pr\{\rv Y \in D \given \rv X =x \}\]
for $\Pr\{\rv Y \in D \given \rv X \in \{x\} \}$.

This elementary notion of conditioning is undefined when the conditioning event has zero measure, such as when a \emph{continuous} random variable
takes a particular value.  
In the modern formulation of conditional probability due to Kolmogorov \citeyearpar{MR0362415},
one defines conditioning with respect to (the $\sigma$-algebra generated by) a \emph{random variable} rather than an individual event. 
In theory, this yields a consistent solution to the problem of conditioning on the value of general (and in particular, continuous) random variables,
although we will see that other issues arise. (See Kallenberg~\citeyearpar[][Ch.~6]{MR1876169} for a rigorous treatment.)

In order to bridge the divide between the elementary notion of conditioning on events and the abstract approach of conditioning on random variables, consider the case of conditioning on a random variable $\rv X$ taking values in a countable discrete set $S$ and satisfying $\Pr \theset {\rv X = x} > 0$ for all $x \in S$.  
Let $\{ \rv Y \in B \}$ be an event.
Then the conditional probability that $\rv Y \in B$ given $\rv X$, written $\CondProb {\rv Y \in B} {\rv X}$, is 
the \emph{random variable}
satisfying $\CondProb {\rv Y \in B} {\rv X} = \Pr \{ \rv Y \in B \given \rv X = x \}$ when $\rv X = x$.
Note that there is a measurable function $f_B\colon S \to [0,1]$ satisfying
\[
	\label{abstracteq}
\Pr \theset{ \rv Y \in B, \ \rv X \in A}  = \int_{A} f_B(x) \,\Pr_{\rv X}(\dee x)
\]
for all measurable subsets $A \subseteq S$.
For sets $A$ of the form $\{x\}$, for $x \in S$,
we have
$\Pr \theset{\rv Y \in B, \ \rv X = x} = f_B(x) \Pr\{\rv X = x\}$, hence $f_B(x) = \Pr \{ \rv Y \in B | \rv X = x \}$.
In summary,
$\CondProb{\rv Y \in B}{\rv X} = f_B(\rv X)$,
and so \eqref{abstracteq}
yields a more abstract characterization of elementary conditional probability for positive-measure events.

The general case is captured by the same defining property.
Let $X$ be a random variable in a measurable space $S$.
Then the \defn{conditional probability that $\rv Y \in B$ given $\rv X$},
written $\CondProb{\rv Y \in B}{\rv X}$,
is defined to be a random variable in $[0,1]$ of the form $f_B(\rv X)$ where 
again $f_B\colon S \to [0,1]$ is such that
\eqref{abstracteq} holds for all measurable subsets $A \subseteq S$.
In many situations, such a function $f_B$ 
is itself the object of interest and so we will let $\CondProbFunc{\rv Y \in B}{\rv X}$ denote an arbitrary such function.  We may then re-express 
its defining property
in the following more intuitive form:
\[\label{intdefnofcondprob}
\Pr \theset{\rv Y \in B,\ \rv X \in A}  = \int_{A} \CondProbFuncEval{\rv Y \in B}{\rv X}{x}\, \Pr_{\rv X}(\dee x)
\]
for all measurable subsets $A \subseteq S$.

The existence of the conditional probability $\CondProb{\rv Y \in B}{\rv X}$, or equivalently, the existence of \linebreak
$\CondProbFunc{\rv Y \in B}{\rv X}$, 
follows from the Radon--Nikodym theorem.  Recall that a measure $\mu$ on a measurable space $S$ is \defn{absolutely continuous} with respect to another measure $\nu$ on the same space, written $\mu \ll \nu$, 
if $\nu(A) = 0$ implies $\mu(A) = 0$ for all measurable sets $A \subseteq S$.
\begin{theorem}[Radon--Nikodym]\label{rnthm}
Let $S$ be a measurable space and let $\mu$ and $\nu$ be $\sigma$-finite measures on $S$ such that $\mu \ll \nu$.  Then there exists a nonnegative measurable function $\RND \mu \nu \colon S \to \NNReals$ such that
\[\label{rndefn}
	\mu (A) = \int_A \RND \mu \nu \, \dee\nu
\]
	for all measurable subsets $A \subseteq S$.
\qed
\end{theorem}
We call any function $\RND \mu \nu$ satisfying Equation~\eqref{rndefn} for all measurable subsets $A \subseteq S$ a \defn{Radon--Nikodym derivative} (of $\mu$ with respect to $\nu$).  

Note that if $g$ is also a Radon--Nikodym derivative of $\mu$ with respect to $\nu$, then $g = \RND \mu \nu$ 
outside a $\nu$-null set, and so \emph{Radon--Nikodym derivatives are unique up to a null set}.
 (Functions that agree a.e.\ are called \defn{versions}.)
We may safely refer to \emph{the} Radon--Nikodym derivative when we want to ignore such differences, but in some cases these differences are important.

It is straightforward to verify that the function $\CondProbFunc{\rv Y \in B}{\rv X}$ is a Radon--Nikodym derivative of 
$\Pr \{\rv Y \in B, \ \rv X \in \pars\}$
with respect to 
$\Pr \theset {\rv X \in \pars} = \Pr_{\rv X}$,
both considered as measures on $S$, 
and so a function $f_B$ satisfying \eqref{abstracteq} for all measurable subsets $A \subseteq S$
always exists, but it is only defined up to a null set.  This is inconsequential when the conditional probability
$\CondProb{\rv Y \in B}{\rv X}$ is the object of interest.
In applications, especially statistical ones, however, 
the function $\CondProbFunc{\rv Y \in B}{\rv X}$ mapping values in $S$ to probabilities is the object of interest,
and, moreover, one typically wants to evaluate this function at particular observed values $x \in S$.
Because
 $\CondProbFunc{\rv Y \in B}{\rv X}$ 
 is merely determined up to a $\Pr_{\rv X}$-null set,
 interpreting its values at individual points is problematic.

As mentioned in the introduction, 
the fact that general conditional probabilities are not uniquely defined at points 
is the subject of a large literature.
However, in some circumstances, two versions of $\CondProbFunc{\rv Y \in B}{\rv X}$ must agree at individual points.  
In particular, if two versions are continuous at a point in the support of the distribution $\Pr_{\rv X}$,
then they agree on the value at that point. In order to state this claim formally, we first recall the definition of the support of a distribution:

\begin{definition}
Let $\mu$ be a measure on a topological space $S$ with open sets $\mathcal S$.  Then the \defn{support of $\mu$}, written $\supp(\mu)$, is defined to be the set of points $x\in S$ such that all open neighborhoods of $x$ have positive measure, i.e.,
\[ \supp(\mu) \defas \{ x \in S \st \forall B \in \mathcal S\  ( x \in B
\implies \mu(B)>0)\}.\]
\end{definition}
Note that the support of $\mu$ can equivalently be defined as the smallest closed set of $\mu$-measure one.
We now state our claim formally:
\begin{lemma}\label{supportlemma}
	Let $S$ be a Polish space.
Suppose $f_1,f_2 \colon S \to [0,1]$ satisfy
$\CondProb{\rv Y \in B}{\rv X} = f_1(\rv X) = f_2(\rv X)$ a.s.
If $x \in S$ is a point of continuity of $f_1$ and $f_2$, and 
$x \in \supp(\Pr_{\rv X})$, then $f_1(x) = f_2(x)$.  In particular, if $f_1$ and $f_2$ are continuous on 
a $\Pr_{\rv X}$-measure one set $D \subseteq S$, then they agree everywhere in $D \cap \supp(\Pr_{\rv X})$.
\qed
\end{lemma}

The proof is immediate from the following elementary result. 
\begin{lemma}\label{equalversions}
Let $f_1, f_2\colon S \to T$ be two measurable functions between Polish spaces $S$ and $T$,
and suppose that $f_1=f_2$ almost everywhere with respect to some measure $\mu$ on $S$.   
Let $D\subseteq S$ be a set of $\mu$-measure one.
If $x \in S$ is a point of continuity of $f_1$ and $f_2$ on $D$,
and 
$x \in \supp(\mu)$, then $f_1(x) = f_2(x)$.
In particular, if $f_1$ and $f_2$ are continuous on $D$,
then they agree everywhere in $D \cap \supp(\mu)$.
\end{lemma}
\begin{proof}
	Let $\delta_T$ be any metric under which $T$ is complete.
Define the measurable function $g\colon S\to\Reals$ by
\[g(x)=\delta_T\bigl(f_1(x), f_2(x)\bigr).\]
We know
that $g = 0$ $\mu$-a.e., and also that $g$ is continuous at
$x$ on $D$, because $f_1$ and $f_2$ are continuous at $x$ on $D$ and $\delta_T$ is
continuous (on all of $T$).  Assume, for the purpose of contradiction, that
$g(x) = \varepsilon > 0$.
By the continuity of $g$ on $D$, there is an open neighborhood $B$ of $x$ such that
$g(B\cap D) \subseteq (\frac \varepsilon 2, \frac {3\varepsilon} 2)$.
But $x \in \supp(\mu)$, hence $\mu(B\cap D) = \mu(B) > 0$,
contradicting $g=0$ $\mu$-a.e.
\end{proof}

The observation that continuity gives a unique answer to conditioning on zero-measure events of the form $\{\rv X = x\}$ is an old one, going back to at least Tjur \citeyearpar{MR0345151}.

\subsection{Conditional Distributions}

For a pair of random variables $\rv X$ and $\rv Y$ taking values in a pair of measurable space $S$ and $T$, respectively, 
it is natural to consider not just individual conditional probabilities $\CondProb{\rv Y \in B}{\rv X}$, for measurable subsets $B \subseteq T$,
but the entire \defn{conditional distribution} 
$\CondProb{\rv Y}{\rv X} \defas \CondProb{\rv Y \in \pars}{\rv X} $.  Unfortunately, the fact that
Radon--Nikodym derivatives are only defined up to a null set
can cause problems.
In particular, while it is the case that 
\[
{\textstyle \sum_j } \CondProb{\rv Y \in B_j}{\rv X} =  \CondProb{\rv Y \in B}{\rv X} \ \as
\]
for every countable measurable partition $B_0,B_1,\dots$ of a measurable set $B \subseteq T$, the random set function given by $B \mapsto \CondProb{\rv Y \in B}{\rv X}$ need not be a measure in general because the exceptional null set may depend on the sequence.  However, when $T$ is Polish, we can construct versions of the conditional probabilities that combine to produce a measure.
In order to make this definition precise, we recall the notion of a probability kernel. 

\begin{definition}[Probability kernel]
Let $S$ and $T$ be Polish spaces.
A function $\kappa\colon S \times \Borel_T \to [0,1]$ is called a \defn{probability kernel (from $S$ to $T$)} when
\begin{enumerate}
\item for every $s \in S$, the function $\kappa(s,\pars)$ is a probability measure on $T$; and
\item for every $B \in \Borel_T$, the function $\kappa(\pars,B)$ is measurable.
\end{enumerate}
\end{definition}
For every $\kappa\colon S \times \Borel_T \to [0,1]$, let $\curry \kappa$ be the map $s \mapsto \kappa(s,\pars)$.
It can be shown that $\kappa$ is a probability kernel from $S$ to $T$ if
and only if $\curry \kappa$ is a (Borel) measurable 
function from
$S$ to $\ProbMeasures(T)$ {\citep[][Lem.~1.40]{MR1876169}}, where we adopt the weak topology on $\ProbMeasures(T)$, which is Polish because $T$ is.

We say that a conditional distribution $\CondProb{\rv Y}{\rv X} $
\defn{has a regular version} when, 
for some probability kernel $\kappa$ from $S$ to $T$,
\[\label{regdefn}
\CondProb{\rv Y \in B}{\rv X}  = \kappa(\rv X, B) \ \as
\]
for every measurable subset $B \subseteq T$.  In this case, we would say that $\curry \kappa(\rv X)$ is a regular version of the conditional distribution.  

\begin{proposition}[Regular versions {\citep[][Lem.~6.3]{MR1876169}}] \label{conddef}
Let $\rv X$ and $\rv Y$ be random variables in a Polish space $S$ and a measurable space $T$, respectively.
Then there is a regular version of the conditional distribution $\CondProb{\rv Y}{\rv X} $, which is, moreover, determined by the joint distribution of $\rv X$ and $\rv Y$.
\qed
\end{proposition}

As with the derivatives underlying conditional probabilities, $\curry\kappa$ is only defined up to a $\Pr_{\rv X}$-null set.  When such a kernel $\kappa$ exists, i.e., when there is a regular version of the conditional distribution $\CondProb{\rv Y}{\rv X}$, we define $\CondProbFunc{\rv Y}{\rv X}$ to be equal to some arbitrary version of $\curry \kappa$.

Despite the fact that the kernels underlying regular versions of conditional distributions are defined only up to sets of measure zero,
it follows immediately from Lemma~\ref{equalversions} that when $S$ and $T$ are Polish,
any two versions of 
$\CondProbFunc{\rv Y}{\rv X}$
that are continuous on some subset of the support of $\Pr_{\rv X}$ must agree on that subset. 
More carefully, let $\curry \kappa_1(\rv X)$ and $\curry \kappa_2(\rv X)$ be regular versions of the conditional distribution 
$\CondProb{\rv Y}{\rv X} $.
If $x \in S$ is a point of continuity
of $\curry \kappa_1$ and $\curry \kappa_2$, and
$x \in \supp(\Pr_{\rv X})$, then $\curry \kappa_1(x) = \curry\kappa_2(x)$.
In particular, if both maps are continuous on 
a set $D \subseteq S$, then they agree everywhere in $D \cap \supp(\Pr_{\rv X})$.

%%%%%%%%%%%%%%%%%%%%%%%%%%%%%%%%%%%%%%%%%%%%%%%%

{When conditioning on a random variable whose distribution concentrates on a countable set, it is well known that a regular version of the
conditional distribution can be built by elementary conditioning with respect to single events.  This includes the special case of conditioning on \defn{discrete random variables}, i.e., 
those concentrating on a countable discrete subspace.

\begin{lemma}\label{discrete-ex} 
Let $\rv X$ and $\rv Y$ be random variables in Polish spaces $S$ and $T$, respectively.
Suppose the distribution of $\rv X$ concentrates on a countable set $R \subseteq S$, i.e., $\Pr_{\rv X}(R) = 1$ and $x \in R$ implies $\Pr_{\rv X}\theset x > 0$.
Let $\nu$ be an arbitrary probability measure on $T$.  Define the function $\kappa \colon S \times \Borel_T \to [0,1]$
by
\[\label{kdef}
\kappa(x,B) \defas \Pr \{\rv Y \in B \given \rv X = x\}
\]
for all $x \in R$ and $\curry \kappa(x) = \nu$ for $x \not\in R$.
Then $\kappa$ is a probability kernel and $\curry \kappa(\rv X)$ is a regular version of the conditional distribution $\CondProb {\rv Y }{\rv X}$.
\end{lemma}
%%%%%%%%%%%%%%%%%%%%%%%%%%%%%%%%%%%%%%%%%%%%%%%%
\begin{proof}
The function $\kappa$ is well-defined because $\Pr\{\rv X = x\} > 0$ for all $x \in R$.
It follows that $\curry \kappa(x)$ is a probability measure for every $x$.  Because $R$ is countable, $\curry \kappa$ is also measurable and so $\kappa$ is a probability kernel from $S$ to $T$.  Note that $\Pr\{ \rv X
\in R\} = 1$ and so, for all measurable sets $A \subseteq S$ and $B \subseteq T$, we have
\[
\int_{A} \kappa(x,B) \, \Pr_{\rv X}(\dee x)
&= \sum_{x\in R \cap A} \Pr \{\rv Y \in B \given \rv X = x\} \, \Pr\{\rv X = x\}\\
&= \sum_{x\in R \cap A} \Pr \{\rv Y \in B,\ \rv X = x\}\\
&= \Pr \{\rv Y \in B,\ \rv X \in A \}.
\]
That is, $\kappa(\rv X, B)$ is the conditional probability of the event $\theset {\rv Y \in B}$ given $\rv X$, and so
$\curry \kappa(\rv X)$ is a regular version of the conditional distribution $\CondProb {\rv Y }{\rv X}$.
\end{proof}

\subsection{Dominated Families}
\label{sec3dom}

Beyond the setting of conditioning on discrete random variables,
explicit formulas for conditional distributions are also available when Bayes' rule applies.
We begin by introducing the notion of a dominated kernel. (The usual terms, such as dominated families or models, refers to measurable families of probability measures, i.e., probability kernels.)

\begin{definition}[dominated kernel]
A probability kernel $\kappa$ from $T$ to $S$
is \defn{dominated} when there is a $\sigma$-finite measure $\nu$ on $S$ such that $\curry\kappa(t) \ll \nu$ for every $t \in T$.
\end{definition}

Let $\rv X$ and $\rv Y$ be random variables in Polish spaces $S$ and $T$, respectively, and let $\curry \kappa_{\rv X | \rv Y}(\rv X)$ be a regular version of $\CondProb {\rv X }{\rv Y}$ such that $\kappa_{\rv X | \rv Y}$ is dominated.  Then there exists a (product) measurable function $p_{\rv X | \rv Y}(\spars|\spars) \colon S \times T \to \Reals^+$
such that $p_{\rv X | \rv Y}(\spars| y)$ is a Radon--Nikodym derivative of
$\curry\kappa_{\rv X | \rv Y}(y)$ with respect to $\nu$
for every $y \in T$,
i.e.,
\[
\label{rndef}
\kappa_{\rv X | \rv Y}(y,A) = \int_A p_{\rv X | \rv Y}(x| y) \, \nu(\dee x)
\]
for every measurable set $A \subseteq S$ and every $y \in T$.

\begin{definition}[conditional density]
We call any such function $p_{\rv X | \rv Y}$ a \defn{conditional density of $\rv X$ given $\rv Y$ (with respect to $\nu$)}.
\end{definition}

Common finite-dimensional, parametric families of distributions (e.g.,\ exponential families like Gaussian, gamma, etc.) are dominated, and  so, in probabilistic models composed from these families, conditional densities exist and Bayes' rule gives a formula for expressing the conditional distribution.
We give a proof of this classic result for completeness.

\begin{lemma}[Bayes' rule {\citep[][Thm.~1.13]{MR1354146}}] \label{bayesrule}
Let $\rv X$ and $\rv Y$ be random variables as in Proposition~\ref{conddef}, 
and assume that there exists a conditional density $p_{\rv X | \rv Y}$ of $\rv X$ given $\rv Y$ with respect to a $\sigma$-finite measure $\nu$.
Let $\mu$ be an arbitrary distribution on $T$ and define $\kappa \colon S \times \Borel_T \to [0,1]$ by
\[
\label{formalbayes}
\kappa(x,B)
=
\frac{\int_{B} p_{\rv X | \rv Y}(x | y) \, \Pr_{\rv Y}(\dee y)}
     {\int_{\phantom{B}}   p_{\rv X | \rv Y}(x | y) \, \Pr_{\rv Y}(\dee y)}, \qquad B \in \Borel_T,
%,
\]
for those points $x \in S$ where the denominator is positive and finite, and by $\curry\kappa(x) = \mu$ otherwise.
Then $\kappa$ is a probability kernel and $\curry\kappa(\rv X)$ is a regular version of the conditional distribution $\CondProb {\rv Y }{\rv X}$.
\end{lemma}
%%%%%%%%%%%%%%%%%%%%%%%%%%%%%%%%%%%%%%%%%%%%%%%%

\begin{proof}
Let $\curry\kappa_{\rv X | \rv Y}(\rv Y)$ be a regular version of the conditional distribution $\CondProb {\rv X }{\rv Y}$.  By hypothesis,  
$\kappa_{\rv X | \rv Y}$ is dominated by $\nu$ and $p_{\rv X | \rv Y}$ is a conditional density with respect to $\nu$.
By Proposition~\ref{conddef} and Fubini's theorem, for measurable sets $A \subseteq S$ and $B \subseteq T$, we have that
\[
\Pr\{\rv X \in A,\ \rv Y \in B\}
&= \int_B \kappa_{\rv X | \rv Y}(y,A) \, \Pr_{\rv Y}(\dee y)\\
&= \int_B \left (\int_A p_{\rv X | \rv Y}(x| y) \,\nu(\dee x) \right)  \Pr_{\rv Y}(\dee y) \\
&= \int_A \left (\int_B p_{\rv X | \rv Y}(x| y) \,\Pr_{\rv Y}(\dee y) \right) \nu(\dee x). \label{finaleq}
\]
Taking $B=T$, we have
\[\label{marg122}
\Pr_{\rv X}(A)
= \int_A \left (\int p_{\rv X | \rv Y}(x| y) \, \Pr_{\rv Y}(\dee y) \right) \nu(\dee x).
\]
Because $\Pr_{\rv X}(S) = 1$, this implies that the set of points $x$ for which the denominator of the right-hand side of \eqref{formalbayes} is infinite has $\nu$-measure zero, and thus $\Pr_{\rv X}$-measure zero.  
Taking $A$ to be the set of points $x$ for which the denominator is zero,
we see that $\Pr_{\rv X}(A) = 0$.
It follows that \eqref{formalbayes} characterizes $\kappa$ up to a $\Pr_{\rv X}$-null set.

By \eqref{marg122}, we see that the denominator is a density of $\Pr_{\rv X}$ with respect to $\nu$, and so we have
\[
\int_A & \kappa(x, B) \, \Pr_{\rv X}(\dee x)
=
\int_A \kappa(x,B) \left (\int p_{\rv X | \rv Y}(x| y) \, \Pr_{\rv Y}(\dee y) \right) \nu(\dee x),
\]
for all measurable sets $A \subseteq S$ and $B \subseteq T$.
Finally, by the definition of $\kappa$, Equation~\eqref{finaleq}, and the fact that the denominator is positive and finite for $\Pr_{\rv X}$-almost every $x$, we see that  $\curry\kappa(\rv X)$ is a regular version of the conditional distribution 
$\CondProb {\rv Y }{\rv X}$.
\end{proof}

%%%%%%%%%%%%%%%%%%%%%%%%%%%%%%%%%%%%%%%%%%%%%%%%

Comparing Bayes' rule
\eqref{formalbayes}
to the definition of conditional density
\eqref{rndef},
we see that
any conditional density of $\rv Y$ given $\rv X$ (with respect to
$\Pr_{\rv Y}$) satisfies
\[
p_{\rv Y | \rv X}(y|x) =
\frac{p_{\rv X | \rv Y}(x | y)}{\int   p_{\rv X | \rv Y}(x | y) \,\Pr_{\rv Y}(\dee y)},
\]
for $\Pr_{(\rv X,\rv Y)}$-almost every $(x,y)$.

The following result suggests why the mere a.e.\ definedness of conditional distributions 
can be ignored by those working entirely within the framework of dominated families.

\begin{proposition}\label{denstokern}
Let $\rv X$ and $\rv Y$ be random variables on Polish spaces $S$ and $T$, respectively, 
let $\curry\kappa(\rv X)$ be a regular version of the conditional distribution $\CondProb{\rv Y}{\rv X}$,
and let $R \subseteq S$.
If a conditional density
$p_{\rv X | \rv Y}(x | y)$ of $\rv X$
given $\rv Y$ is continuous on $R \times T$, positive, and bounded,
then $\kappa$ as defined in \eqref{formalbayes} is a version of $\curry\kappa$ that is continuous on $R$.
In particular, if $R$ is a $\Pr_{\rv X}$-measure one subset, then $\kappa$ is a $\Pr_{\rv X}$-almost continuous version.
\end{proposition}

We defer the proof to Section~\ref{cds}.
We will use this result in the proof of Lemma~\ref{lem:effectiveversion}, towards 
our central result.

%%%%%%%%%%%%%%%%%%%%%%%%%%%%%%%%%%%%%%
%%%%%%%%%%%%%%%%%%%%%%%%%%%%%%%%%%%%%%
%%%%%%%%%%%%%%%%%%%%%%%%%%%%%%%%%%%%%%
%%%%%%%%%%%%%%%%%%%%%%%%%%%%%%%%%%%%%%

\section{Computable Conditional Probabilities and Distributions}
\label{ccd2}

Before we lay the foundations for the remainder of the paper and define notions of computability for conditional probability and conditional distributions in the abstract setting, we address the computability of distributions conditioned on positive-measure sets.
In order for the distributions 
obtained from positive measure sets
to be computable,
we will need the conditioning events to be almost decidable sets.

\begin{lemma}[{\citep[][Prop.~3.1.2]{MR2558734}}]
\label{condcontset}
Let $(S,\mu)$ be a computable probability space and let $A$ be a
$\mu$-almost
decidable subset of $S$ satisfying $\mu(A)>0$.
Then $\mu(\pars|A)$ is a
computable probability measure, uniformly in a witness to the
$\mu$-almost decidability of $A$.
\end{lemma}
\begin{proof}
By Lemma~\ref{almdecbasis} there is a $\mu$-almost decidable basis
for $S$. Note that $\mu(\pars|A)$ is absolutely continuous with respect to
$\mu$. Hence
by Corollary~\ref{almostcomplementary-to-compmeas}, it suffices to show that
$\frac{\mu(B \cap A)}{\mu(A)}$
is computable for a $\mu$-almost decidable set $B$, uniformly in 
witnesses to the $\mu$-almost decidability of $A$ and $B$. All subsequent
statements in this proof are uniform in both.
Now,
$B \cap A$ is $\mu$-almost decidable 
with computable witness,
and so its measure, the numerator, is a computable  real.
The denominator is likewise the measure of a set that is almost
decidable with computable witness,
hence is a computable real.
Finally, the ratio of two computable reals is itself computable.
\end{proof}

In the abstract setting, conditional probabilities are random variables.
In many applications of probability, including statistics,
the conditional probability map, or some version of it, is the actual object of interest, and so the computability of this map is our focus.

Let $B \subseteq T$ be a measurable set.
Viewing $\CondProbFunc{\rv Y \in B}{\rv X}$ as a function from $S$ to $[0,1]$, 
recall that we can speak formally as to whether this function is
everywhere computable, $\Pr_{\rv X}$-almost computable, and/or $L^1$-computable.
Recall also that the function $\CondProbFunc{\rv Y \in B}{\rv X}$ may have many versions that agree only up to a null set.  Despite this, their almost computability does not differ (up to a change in domain by a null set).
\begin{lemma}
Let $f$ be a measurable function from a computable probability space $(S,\mu)$ to a computable Polish space $T$.
If any version of $f$ is computable on a $\mu$-measure $p$ set, then every version of $f$ is computable on a $\mu$-measure $p$ set.  In particular, if one version is $\mu$-almost computable, then all version are.
\end{lemma}
\begin{proof}
Let $f$ be computable on a $\mu$-measure $p$ set $D$, and let $g$ be a version of $f$, i.e.,
\linebreak
$Z \defas \theset{ s \in S \st f(s) \neq g(s) }$ 
is a $\mu$-null set.  Therefore, $f=g$ on
$D \setminus Z$.  Hence $g$ is computable on the $\mu$-measure $p$ set $D \setminus Z$.  If $f$ is $\mu$-almost computable, then it is computable on a $\mu$-measure one set, and so $g$ is as well.
\end{proof}

We can develop notions of computability for conditional distributions in a similar way.  We begin by characterizing the computability of probability kernels.

\begin{definition}[Computable probability kernel]
\label{compkern}
Let $S$ and $T$ be computable Polish spaces and let $\kappa \colon S
\times \Borel_T \to [0,1]$ be a probability kernel from $S$ to $T$.
Then we say that $\kappa$ is a \defn{computable probability
kernel} when $\curry \kappa \colon S \to \ProbMeasures(T)$ given by
$\curry \kappa(s) \defas \kappa(s,\pars)$ is a computable function in the ordinary sense between $S$ and the computable Polish space $\ProbMeasures(T)$ induced by $T$.
Similarly, we say that $\kappa$ is computable on a subset $D
\subseteq S$ when $\curry \kappa$ is computable on $D$.
\end{definition}

As we will see, this notion of computability corresponds with a more direct notion of computability for $\kappa$, which we now develop. We begin by noting that the collection of sets of the form 
\[
\label{PaqDef}
      \ProkBallS{T}{A, q} \defas \{ \mu \in \ProbMeasures(T) \st \mu(A) > q \}
\]
for $A$ open and $q$ rational,
form a subbasis for the weak topology on $\ProbMeasures(T)$ (which is the topology induced by the Prokhorov metric).  Indeed, it suffices for $A$ to range over finite unions of some countable basis of $T$. We will also omit mention of $T$ when the ambient space is clear from context. 

The next result relates balls in the Prokhorov metric to the subbasis elements above.
Recall that $\delta_p$ denotes the Prokhorov metric and that the collection $ \cD_P$ of measures with finitely many point masses on elements $\cD_T$, each assigned rational mass, form a dense set.  
\begin{proposition}[{\citep[][Prop.~B.17]{MR2159646}}]
\label{Prokhorov-ball-condition}
Let $\nu,\mu \in \ProbMeasures(T)$, and assume that $\nu$ is supported on a finite set $S$.
Then the condition $\delta_p(\nu,\mu) < \epsilon$ is equivalent to the finite set of conditions
\[
\mu(A^\epsilon) > \nu(A) - \epsilon
\]
for all $A \subseteq S$.
\qed
\end{proposition}

The next corollary states that we can compute a representation for a Prokhorov ball in terms of the subbasis elements. The sets are easily defined from those in Proposition~\ref{Prokhorov-ball-condition}.

\begin{corollary}
\label{Prokhorov-ball-in-terms-of-P-A-q} 
Uniformly in $\nu \in \cD_P$ and $\epsilon \in \Rationals$, we can compute a finite collection of pairs $(A_i, q_i)_{i \leq n}$, each $A_i$ a finite union of open balls of radius $\epsilon$ around elements of $\cD_T$ and each $q_i$ a rational, 
such that 
\[
\{\mu \in \ProbMeasures(T)\st \delta_p(\mu, \nu) < \epsilon\} = \bigcap_{i \leq n}\ProkBall{A_i, q_i}.
\qed
\]
\end{corollary}

Finally, as a direct consequence of \citep[][Prop.~4.2.1]{MR2519075}, these subbasis elements are c.e.\ open.

\begin{proposition}
\label{equiv-weak-top-and-metric-for-Prokhorov}
Let $A$ be a c.e. open subset of $T$ and $q$ be a rational. 
Then the set $\ProkBall{A,q}$ is c.e.\ open in the Prokhorov metric, uniformly in $A$ and $q$. 
\qed
\end{proposition}

Recall that a lower semicomputable function from a computable Polish
space to $[0,1]$ is one for which
the preimage of $(q, 1]$
is c.e.\ open, uniformly in rationals $q$.
Furthermore, we say that a function $f$ from a computable Polish space $S$ to $[0,1]$ is \emph{lower semicomputable on $D\subseteq S$} when
there is a uniformly computable sequence
$\{U_q\}_{q\in\Rationals}$
of c.e.\ open sets
such that
\[f^{-1}\bigl[(q,1]\bigr] \cap D = U_q \cap D.\]

We can also interpret a
computable probability kernel
$\kappa$
as a computable map sending each c.e.\ open set $A \subseteq T$ to a lower semicomputable function
$\kappa(\pars,A)$.

\begin{lemma}\label{phikappa}
Let $S$ and $T$ be computable Polish spaces,
let $\kappa$ be a probability kernel from $S$ to $T$, and let $D \subseteq S$.
If
$\curry \kappa$ is computable on $D$
then
$\kappa(\pars,A)$ is lower semicomputable on $D$ uniformly in the c.e.\
open set $A$, and conversely.
\end{lemma}
\begin{proof}
Let $q \in (0,1)$ be rational, 
let $A\subseteq T$ be c.e.\ open,
and define $I \defas (q,1]$.
Then 
\[
\label{invEq}
\kappa^{-1}\bigl(\cdot,A\bigr)[I] = \{ x\st \curry \kappa(x)(A) \in I\} = \curry \kappa^{-1}[\ProkBall{A,q}],
\]
where $\ProkBall{A,q}$ is as in \eqref{PaqDef}.
By Proposition~\ref{equiv-weak-top-and-metric-for-Prokhorov},
$\ProkBall{A,q}$ is even c.e.\ open.

Suppose $\curry \kappa$ is computable on $D$. Then
there is a c.e.\ open set $V_{A,q}$, uniformly computable in $q$ and $A$, such that
\[
\label{invIntersected}
V_{A,q} \cap D = \curry  \kappa^{-1}[\ProkBall{A,q}] \cap D =
\kappa(\pars,A)^{-1}[I] \cap D,
\]
and so $\kappa(\pars,A)$ is lower semicomputable
on $D$, uniformly in $A$.

Conversely, suppose
$\kappa(\pars,A)$ is lower semicomputable
on $D$, uniformly in $A$. 
Then by \eqref{invEq}, uniformly in $A$ and $q$, we can find a c.e.\ open
$V_{A,q}$ such that 
\eqref{invIntersected} holds. 

By Corollary~\ref{Prokhorov-ball-in-terms-of-P-A-q}, every  basic open ball in the Prokhorov metric is the finite intersection of sets of the form $\ProkBall{A, q}$,
which are c.e.\ open themselves because a finite intersection of c.e.\ open sets is c.e.\ open. Therefore, uniformly in a c.e.\ open set $U$ in the Prokhorov metric, 
we can find a c.e.\ open set $V$ in $S$ such that 
\[
\label{invIntersectedb}
V \cap D = \curry  \kappa^{-1}[U] \cap D.
\]
Hence $\curry \kappa$ is computable on $D$.
\end{proof}

%%%%%%%%%%%%%%

Let $\rv X$ and $\rv Y$ be random variables in computable Polish spaces $S$ and $T$, respectively, and let $\curry\kappa(\rv X)$ be a regular version of the conditional distribution $\CondProb{\rv Y}{\rv X}$. 
The above notions of computability are suitable for talking about the computability of $\kappa$ or any other version of it,
and are appropriate notions of computability for statistical applications.

Intuitively, a probability kernel $\kappa$ is computable when,
for some (and hence for any) version of $\kappa$,
there is a program that, given as input a representation of a point
$s \in S$, outputs a representation of the measure $\curry\kappa(s)$ for $\Pr_{\rv X}$-almost every input $s$.

%%%%%%%%%%%%%%%%%%%%%%

%%%%%%%%%%%%%%%%%%%%%%%%%%%%%%%%%%%%%%
%%%%%%%%%%%%%%%%%%%%%%%%%%%%%%%%%%%%%%
%%%%%%%%%%%%%%%%%%%%%%%%%%%%%%%%%%%%%%
\section{Discontinuous Conditional Distributions}
\label{Sec:discontinuous}

Our study of
the computability of conditional distributions
begins at the following roadblock:
a conditional
distribution need not have \emph{any} version
that is continuous or even almost continuous (in the sense described
in Section~\ref{ccd2}).  This will rule out almost computability (though not $L^1$-computability).

We will work with
the standard effective presentations of the spaces $\Reals$, $\Naturals$,
$\{0,1\}$, as well as product spaces thereof, as computable Polish spaces.
For example, we will use
$\Reals$ under the Euclidean metric, along with
the ideal points $\Rationals$
under their standard enumeration.

Recall that a random variable $\rv C$ is a
\defn{Bernoulli}($p$) random variable
when $\Pr\{\rv C = 1\} =$ \linebreak $1 - \Pr\{\rv C = 0\} = p$.
A random variable $\rv N$ is a \defn{geometric}($p$)  random variable
when it takes values in $\Naturals = \{0,1,2,\dotsc\}$ and satisfies
\[
\Pr\{ \rv N = n \} = p^n\, (1-p)
\]
for all $n \in \Naturals$.
A random variable that takes values in a finite set is \defn{uniformly distributed} when it assigns equal probability to each element.  A continuous random variable $\rv U$ on the unit interval is \defn{uniformly distributed} when the probability that it falls in the subinterval $[\ell,r]$ is $r-\ell$.  It is easy to show that the distributions of these random variables are computable, provided $p$, $\ell$, and $r$ are computable reals and $ p \in [0,1]$.

Let $\rv N$, $\rv C$, and $\rv U$ be independent $\Pr$-almost computable random variables such that
$\rv N$ is a geometric($\frac12$) random variable,
$\rv C$ is a Bernoulli($\frac12$) random variable, and
$\rv U$ is a uniformly distributed random variable in $[0,1]$.
Fix a computable enumeration $\{r_i\}_{i\in\Naturals}$ of the rational numbers (without repetition) in $(0,1)$, and
consider the random variable
\[
\rv X \defas
\begin{cases}
r_{\rv N}, & \text{if } \rv C = 1;\\
\rv U, & \text{otherwise},
\end{cases}
\]
which is also $\Pr$-almost computable because it is a computable function of $\rv C$, $\rv U$, and $\rv N$.

\begin{proposition}\label{discontinuouscond}
Every version of $\CondProbFunc{\rv C = 1}{\rv X}$
is discontinuous everywhere on every $\Pr_{\rv X}$-measure one set.
In particular, no version is $\Pr_{\rv X}$-almost computable.
\end{proposition}

\begin{proof}\label{discontinuouscond-proof}
Note that
$
\Pr\{\rv X \text{ rational} \} = \frac12
$
and, furthermore, 
$
\Pr\{ \rv X = r_k \} = \frac1{2^{k+2}} > 0.
$
Therefore, any two versions of $\CondProbFunc{\rv C = 1}{\rv X}$ must agree on \emph{all} rationals in $[0,1]$.
In addition, because $\Pr_{\rv U} \ll \Pr_{\rv X}$, i.e.,
\[\Pr\{\rv U \in A\} > 0 \implies \Pr \{\rv X \in A\} > 0\]
 for all measurable sets $A\subseteq [0,1]$, any two versions must agree on a Lebesgue-measure one set of the irrationals in $[0,1]$.
An elementary calculation shows that
\[\Pr\{ \rv C=1 \given \rv X\ \text{rational} \} = 1,
\]
while
\[\Pr\{ \rv C=1 \given \rv X\ \text{irrational} \} = 0.\]
It is also straightforward to verify that $\rv C$ and $\rv X$ are conditionally independent, 
given an indicator for the event $\theset {\rv X \text{ rational}}$.
Therefore, all versions $f$ of $\CondProbFunc{\rv C = 1}{\rv X}$ satisfy, for $\Pr_{\rv X}$-almost every $x$,
\[
\label{Dirichlet}
f(x) =
\begin{cases}
1, & x \text{\ rational};\\
0, & x \text{\ irrational}.
\end{cases}
\]
The right hand side, considered as a function of $x$, is called the
Dirichlet function, and is \emph{nowhere continuous}.

Suppose some version of $f$ were continuous at a point $y$ on a $\Pr_{\rv X}$-measure one set $R$.  Then there would exist an open interval $I$ containing $y$ such that the image of $I\cap R$ contains 0 or 1, but not both.  However, $R$ must contain all rationals in $I$ and Lebesgue-almost every irrational in $I$.  Furthermore, the image of every rational in $I\cap R$ is 1, and the image of Lebesgue-almost every irrational in $I\cap R$ is $0$, a contradiction.
\end{proof}

Although we cannot hope to compute $\CondProbFunc{\rv C = 1}{\rv X}$ on a $\Pr_{\rv X}$-measure one set, we can compute it in a weaker sense.
\begin{proposition}\label{thm:merelyl1comp}
$\CondProbFunc{\rv C = 1}{\rv X}$ is $L^1(\Pr_{\rv X})$-computable.
\end{proposition}
\begin{proof}
By Corollary~\ref{L1aslimit}, it suffices to construct a sequence of uniformly $\Pr_{\rv X}$-almost computable functions that converge effectively in $L^1(\Pr_{\rv X})$ to
$\CondProbFunc{\rv C = 1}{\rv X}$.
Let $\phi_k = \frac 1 2 \min_{m < n \le k} | r_m - r_n |$
	be half the minimum distance between any pair among $r_0,\dotsc, r_k$,
and define, for every $k \in \Nats$,
\[
f_k(x) \defas \begin{cases}
	1,&  \text{if~~} |x-r_n| < \frac 1 {(k+1)\sqrt{2}} \min ( 2^{-k-2}, \phi_k ) \text{~holds for some $n \le k$;} \\
0, & \text{otherwise.}
\end{cases}
\]
Note that the set on which $f_k$ takes the value 1 is uniformly $\Pr_{\rv X}$-almost decidable in part because its boundary points are irrationals, a null set.
It is then clear that the functions $f_k$, for $k \in \Nats$, are uniformly
$\Pr_{\rv X}$-almost computable.
For every $k \in \Nats$, we have that 
$\Pr \theset { \rv X = r_n \text{ for some $n > k$} } = 2^{-k-2}$.
Therefore,
\[
&\int \Bigl| \CondProb{\rv C = 1}{\rv X=x} - f_k(x) \Bigr| \, \Pr_{\rv X}(\dee x)\\
&\qquad\le \Pr \theset { \rv X = r_n \text{ for some $n > k$} } %\frac 1 2 \sum_{n > k} 2^{-n} 
 + \int_{[0,1]\setminus\Rationals} 
\Bigl| \CondProb{\rv C = 1}{\rv X=x} - f_k(x) \Bigr| \, \Pr_{\rv X}(\dee x) \\
	&\qquad\le 2^{-k-2} + \frac 1 {(k+1)\sqrt{2}} \sum_{n=0}^k 2^{-k-2} 
< 2^{-k-1},
\]
completing the proof.
\end{proof}

%%%%%%%%%%%%%%%%%%%%%%%%%%%%%%%%%%%%%%
%%%%%%%%%%%%%%%%%%%%%%%%%%%%%%%%%%%%%%

%%%%%%%%%%%%%%%%%%%%%%

\section{Conditioning is Discontinuous}
\label{cond-is-discon}

Conditioning in general can produce discontinuous conditional distributions, which is an obstruction to a conditioning operator being computable.  
But even if we restrict our attention to distributions that admit conditional distributions that are continuous on their support, the operation of conditioning cannot be computable because, as we will show, it is discontinuous. 
Indeed, conditioning is discontinuous in a rather strong way.  
We use the recursion theorem to explain the computational consequences.  
Namely, for any potential program analysis that aims to perform conditioning on an arbitrary given distribution, there is a representation of that distribution such that the program analysis cannot identify a single nontrivial fact about its conditional distribution.

To begin, we formalize the notion of a conditioning operator.

\newcommand{\SPM}{\mathcal F}
\begin{definition}
Let $\SPM \subseteq \ProbMeasures([0,1]^2)$ be a set of probability measures.\ \ A
map $\Phi\colon \ProbMeasures([0,1]^2) \times [0,1] \to \ProbMeasures([0,1])$ 
is a \defn{conditioning operator} (for $\SPM$)
if, for all distributions $\mu \in \SPM$ and random variables $\rv X$ and $\rv Y$ with joint distribution $\mu$,
we have $\Phi(\mu,x) = \CondProbFuncEval{\rv Y}{\rv X}{x}$ for $\Pr_{\rv X}$-almost all $x$. 
\end{definition}

Observe, by Proposition~\ref{conddef}, that there is a conditioning operator for all $\SPM \subseteq \ProbMeasures([0,1]^2)$.

\begin{definition}
Let $\SPM \subseteq \ProbMeasures([0,1]^2)$.
	A conditioning operator for $\SPM$ is \defn{computable} if it is computable 
	on $\SPM \times [0,1]$,
	considered as a function $\ProbMeasures([0,1]^2) \times [0,1] \to \ProbMeasures([0,1])$,
where both
 $\ProbMeasures([0,1]^2) \times [0,1]$ and $\ProbMeasures([0,1])$ are taken to be the canonical computable Polish spaces.
\end{definition}

The previous section motivates restricting one's attention to conditioning operators for the set $\SPM_0\subseteq \ProbMeasures([0,1]^2)$ of probability distributions on pairs $(\rv X, \rv Y)$ of random variables in $[0,1]$
such that there exists 
a $\Pr_{\rv X}$-almost continuous version of the conditional distribution map $\CondProbFunc{\rv Y}{\rv X}$. 
We will show that conditioning operators for $\SPM_0$ are not computable, simply on grounds of continuity.

Recall that the name of a probability measure $\mu \in \ProbMeasures(T)$ 
on a computable Polish space $T$
is given by a Cauchy sequence in the dense elements $\cD_{P, T}$ of the associated Prokhorov metric. 
Note that $\SPM_0$ contains $\cD_{P, [0,1]^2}$.
Further recall that
$\ProkBallS{T}{A, q}$ is defined to be the set $\{\eta \in \ProbMeasures(T) : \eta(A) > q\}$,
for any open set $A \subseteq T$ and rational $q \in \Rationals$. 
Let $\cA[T] \defas \{(A, q)\st A$ is a finite union of open balls in $T$, and $q \in \Rationals\}$. 

Given a computable Cauchy sequence in the Prokhorov metric 
that converges to a measure $\mu \in \ProbMeasures(T)$, 
by Corollary~\ref{Prokhorov-ball-in-terms-of-P-A-q} we can compute a sequence $\<A_i, q_i\>_{i \in \Nats}$ in $\cA[T]$ 
such that $\bigcap_{i \in \Nats} \ProkBallS{T}{A_i, q_i} = \{\mu\}$. Further by Proposition \ref{equiv-weak-top-and-metric-for-Prokhorov}, given a finite sequence $\<A_i, q_i\>_{i \leq n}$ in $\cA[T]$ we can compute,
 uniformly in $\<A_i, q_i\>_{i \leq n}$, 
 a ball in the Prokhorov metric contained in $\bigcap_{i \leq n}\ProkBallS{T}{A_i, q_i}$.
Therefore, uniformly in a probability measure $\mu$ and a collection $\<{A_i, q_i}\>_{i \in \Nats}$ with $\bigcap_{i \in \Nats} \ProkBallS{T}{A_i, q_i} = \{\mu\}$,
 we can computably recover a name for $\mu$.
Conversely, from a name for $\mu$, we can uniformly compute such a collection. 

\begin{lemma}
\label{recursion-thm-distribution}
Let $\SPM \subseteq \ProbMeasures([0,1]^2)$ contain $\cD_{P,[0,1]^2}$.
For every $\alpha \in \cD_{P,[0,1]}$, computable representation $\{\nu_i\}_{i \in \Nats}$ of 
$\nu \in \ProbMeasures([0,1]^2])$, computable representation $\{x_i\}_{i \in \Nats}$ of
$x \in [0,1]$, and rational $\epsilon \in (0,1)$, we can uniformly find a measure $\mu \in \cD_{P,[0,1]^2}$ such that
$\delta_P(\mu,\nu) < \epsilon$ and $\Phi(\mu,x) = \alpha$
for every conditioning operator $\Phi$ for $\SPM$.
\end{lemma}
\begin{proof}
Relative to $\{\nu_i\}_{i \in \Nats}$ and $\{x_i\}_{i \in \Nats}$, we can computably find an element $p^* \in \cD_{P,[0,1]^2}$ such that
 $p^*(\{x\}\times [0,1]) = 0$
 and $\delta_P(p^*,\nu) < \frac{\epsilon}{2}$. 
 Let $\mbox{\boldmath$\delta$}_{x}$ denote the Dirac measure on $[0,1]$ concentrating on $x$
 and let $\tau \otimes \tau'$ denote the product measure on $[0,1]^2$ with respective marginal distributions $\tau,\tau' \in \ProbMeasures([0,1])$.
Defining $p \defas 
         \frac\epsilon2(\mbox{\boldmath$\delta$}_{x} \otimes \alpha) + 
	(1-\frac\epsilon2) p^*$,
	it is easy to check that $\delta_P(p,p^*) \le \frac{\epsilon}{2}$, hence $\delta_P(\nu,p) < \epsilon$. 
	Because $p^*(\{x\} \times [0, 1]) = 0$ and $p^*$ is a finite mixture of point masses, 
         every conditioning operator $\Phi$ for $\SPM$ must satisfy $\Phi(\mu,x)= \alpha$.
\end{proof}

\begin{proposition}\label{mainnegativeresult}
Let $\SPM \subseteq \ProbMeasures([0,1]^2)$ contain $\cD_{P,[0,1]^2}$.
Every conditioning operator on $\SPM$ is discontinuous everywhere, hence noncomputable.
\end{proposition}
\begin{proof}
On $\ProbMeasures([0,1])$, adopt the weak topology (induced by the standard topology on $[0,1]$).
On $\ProbMeasures([0,1]^2) \times [0,1]$, adopt the product topology induced by the weak and standard topologies, respectively.
Then Lemma~\ref{recursion-thm-distribution} implies that every conditioning operator $\Phi$ for $\SPM$
is discontinuous everywhere. Hence every conditioning operator is noncomputable.
\end{proof}

The above definitions and proposition capture the essential difficulty of conditioning: a finite approximation to the joint distribution determines nothing about the result of conditioning on a particular point.

We now establish a stronger notion of noncomputability, namely
that it is not even possible to always produce \emph{some} nontrivial fact about a conditional distribution.
For $e\in\Naturals$, let $\varphi_e$ denote the partial computable
function defined by code $e$.
The recursion theorem, due to \citet{Kleene38a},
 states that when $F$ is a total computable function, there is some integer $i$
for which the partial computable functions $\varphi_i$ and
$\varphi_{F(i)}$ are equal partial functions (i.e., they are defined
on the same inputs, and are equal where they are defined).
For more details, see \citep[][Ch.~11]{MR886890}.

\begin{definition}
A program $\varphi_e\colon \Nats \to \Nats$ \defn{represents a distribution} $\mu_e$ on a computable Polish space $T$
if it is total and on input $k$, the output value $\varphi_e(k)$ is a code for a pair $(A_k,q_k) \in \cA[T]$
such that $\{\mu\} =  \bigcap_{k \in \Nats} \ProkBallS{T}{A_k, q_k}$.
\end{definition}

\begin{definition}
A program $\varphi_a\colon \Nats^3 \to \Nats$ is a \defn{conditioning program} for $\SPM \subseteq \ProbMeasures([0,1]^2)$
if it is total and
	whenever $e$ represents a computable distribution $\mu_e \in \SPM$,
there exists a conditioning operator $\Phi$ for $\SPM$
such that, 
for every code $j \in \Nats$ for a computable real $x \in [0,1]$,
and for every $k \in \Nats$,
the return value $\varphi_a(e, j, k)$ is a code for either the empty string 
or an element $(A,q)\in \cA[{[0,1]}]$ such that
	$\Phi(\mu_e,x) \in \ProkBallS{[0,1]}{A, q}$ 
	and 
	$\ProkBallS{[0,1]}{A, q} \neq \ProbMeasures([0,1])$.
\end{definition}

\begin{theorem}[Nonapproximable conditional distributions]
\label{recthmresult}
Suppose that $\varphi_a$ is a conditioning program for a set $\SPM \subseteq \ProbMeasures([0,1]^2)$ containing $\cD_{P,[0,1]^2}$.
Let $e \in \Nats$ be a code for a computable distribution $\mu_e$ on $[0,1]^2$,
and let $j \in \Nats$ be a code for a computable real $x \in [0,1]$.
Then uniformly in $a$, $e$, and $j$, we can compute an $i \in \Nats$ such that $\mu_e = \mu_i$
and 
$\varphi_a(i,j,k)$ is a code for the empty string for every $k \in \Nats$.
\end{theorem}
\begin{proof}
	Uniformly in $a$, we can compute some
	$b\in\Nats$ such that for all $n, m, r \in\Nats$:
	if the value $\varphi_a(n, m, r')$ is a code for the empty string
	for all $r' \le r$, then $\varphi_b(n, m, r)$ is also a code for the empty string;
	and otherwise
	$\varphi_b(n, m, r) = \varphi_a(n, m, r')$, where $r'$ is the least index such that 
	$\varphi_a(n, m, r')$ is not a code for the empty string.
	Note that for each $n, m\in\Nats$, the value $\varphi_b(n,m,k)$ is a code for the empty string for all $k\in\Nats$ if and only if $\varphi_a(n,m,k)$ is a code for the empty string for all $k$.
	Let $\eta_{n,j}$ denote the least index $k \in \Nats$ such that  $\varphi_{b}(n,j,k)$ is not a code for the empty string, if such $k$ exists, and $\infty$ otherwise.
	Note that for each $k\in\Nats$, we can compute (uniformly in $n$, $e$, $a$, and $j$) whether or not $k < \eta_{n,j}$ (even though the finiteness of $\eta_{n,j}$ may not be computable).
	For $k\in\Nats$, let $(A_k, q_k)$ be the pair coded by $\varphi_e(k)$.

Define the total computable function 
$F\colon \Nats \to \Nats$ 
	such that for $n,k\in\Nats$,
	\[
		\label{rec-cases}
		\varphi_{F(n)}(k)  = \begin{cases}
			\varphi_e(k) & \text{if} \quad k < \eta_{n, j};\text{~and}\\
			\varphi_{e'}(k-\eta_{n,j}) & \text{if} \quad k \ge \eta_{n, j},
		\end{cases}
			\]
			where $e'\in\Nats$ is defined as follows.
	Let $(A', q')$ be the pair coded by $\varphi_{b}(n,j,\eta_{n, j})$.
First, compute a Prokhorov ball $B \subseteq \ProbMeasures([0,1]^2)$ 
contained within $\bigcap_{\ell\le \eta_{n, j}} \ProkBallS{[0,1]^2}{A_{\ell}, q_{\ell}}$.
	Next, compute some $\alpha \in \ProbMeasures([0,1])$ such that $\alpha(A') = 0$, and
hence $\alpha \not\in \ProkBallS{[0,1]}{A',q'}$.
Then, by Lemma~\ref{recursion-thm-distribution}, compute a code $e'$ for a distribution $\nu \in B$ such that
$\Phi(\nu,x) = \alpha$
for every conditioning operator $\Phi$ for $\SPM$.

By the recursion theorem,
we can compute an index $i$, uniformly in $a$, $e$ and $j$,
such that $\varphi_{F(i)} = \varphi_i$.
We now argue
that
	$\eta_{i,j} = \infty$,
	which implies that 
	$\varphi_i = \varphi_e$ by \eqref{rec-cases}.

	Suppose, for a contradiction, that $\eta_{i,j} \in\Nats$. Then $\varphi_{b}(i,j,\eta_{i,j}) = (A',q')$ for some $(A', q') \in \cA[{[0,1]}]$. 
Hence, as $\varphi_a$ is a conditioning program, 
	there is some conditioning operator $\Phi$ for $\SPM$, such that $\Phi(\mu_i,x) \in \ProkBallS{[0,1]}{A',q'}$, where $\mu_i$ is the measure represented by $\varphi_i$. 
	By construction, 
for every conditioning operator $\Phi$ for $\SPM$,
we have $\Phi(\mu_i,x) \not \in \ProkBallS{[0,1]}{A', q'}$, a contradiction.
\end{proof}

These results rely on the density of the finitely supported discrete probability distributions $\cD_{P,[0,1]^2}$.  However, analogous results 
can be established if we restrict ourselves to absolutely continuous distributions admitting continuous joint density functions. In this case, the role of the finitely supported continuous distributions would be played by absolutely continuous distributions with sharp but continuous bump functions concentrating on small sets.
The fundamental obstruction is the same: partial information in the weak topology does not suffice to condition continuously.

%%%%%%%%%%%%%%%%%%%%%%%%%%%%%%%%%%%%%%

\section{Noncomputable Almost-Continuous Conditional Distributions}
\label{Sec:c.e. neg}

In this section, we construct a pair of $\Pr$-almost computable random variables $\rv X$ in $[0,1]$ and 
$\rv N$ in $\Naturals$ such that the conditional probability map $\CondProbFunc{\rv N = k}{\rv X}$ is not even $L^1(\Pr_{\rv X})$-computable,
despite the existence of an $\Pr_{\rv X}$-almost continuous version.
Our construction in this section can be thought of as providing a single witness to the noncomputability of the conditioning operator.

Let $M_n$ denote the $n$th Turing machine, under a standard enumeration,
and let ${h\colon \Naturals \to \Naturals \cup \{\infty\}}$ be the map given by
$h(n)\defas\infty$ if $M_n$
does not halt (on
input 0) and $h(n)\defas k$ if
$M_n$
halts (on input 0)
at the $k$th step.
%%%%% %%%%%
We may then take $\zj \colon \Nats \to \{0,1\}$ to denote the halting set 
\[\{ \ell \st M_\ell \text{~halts on input~}0\},\]
which is computably enumerable but not
computable.
The set $\zj$ and the function $h$ are 
computable from
each other because
\[
\zj = \{ n \in \Nats \st h(n) < \infty\}.
\]
We now use $h$ to define a pair of $\Pr$-almost computable random variables $(\rv N, \rv X)$ such that 
$\zj$ 
is computable from $\CondProbFunc{\rv N}{\rv X}$.

Let $\rv N$, $\rv C$, $\rv U$, and $\rv V$ be independent $\Pr$-almost computable random variables such that
$\rv N$ is a geometric($\frac15$) random variable, 
$\rv C$ is a Bernoulli($\frac13$) random variable, and
$\rv U$ and $\rv V$ are uniformly distributed random variables in $[0,1]$.

Let $\lfloor x \rfloor$ denote the greatest integer $y \le x$,
and note that $\lfloor 2^k \rv V \rfloor$ is
uniformly distributed in $\{0, 1, 2, \dotsc, 2^k -1 \}$ and is $\Pr$-almost computable. 
For each $k \in \Nats$, consider the derived random variable
\[
\rv X_k \defas \frac{2 \lfloor 2^k \rv V \rfloor + \rv C + \rv U}{2^{k+1}}\ .
\]
Note that $\lim_{k\to\infty} \rv X_k$ almost surely exists.
Define $\rv X_{\infty} \defas \lim_{k\to\infty} \rv X_k$, and observe that 
$\rv X_\infty  = \rv V \ \textrm{a.s.}$
Finally, define
$
\rv X \defas \rv X_{h(\rv N)}.
$

%%%%%%%%%%%%%%%%%%%

\begin{proposition}\label{thmcomp}
The random variable $\rv X$ is $\Pr$-almost computable.
\end{proposition}
\begin{proof}
Because $\rv U$ and $\rv V$ are computable on a $\Pr$-measure one set and a.s.\ nondyadic,
their 
binary expansions $\{\rv U_n \st n \in \Nats\}$ and $\{\rv V_n \st n \in \Nats\}$
(which are uniquely determined with probability 1)
are themselves $\Pr$-almost computable random variables in $\{0,1\}$, uniformly in $n$.

For each $k \ge 0$, define the random variable
\[
\rv D_k = \begin{cases}	
\rv V_k,           & h(\rv N) > k; \\
\rv C,             & h(\rv N) = k; \\
\rv U_{k-h(\rv N)-1}, & h(\rv N) < k.
\end{cases}
\]
By simulating $M_{\rv N}$ for $k$ steps, we can decide whether
$h(\rv N)$ is less than, equal to, or greater than $k$.
Therefore the random variables
$\{\rv D_k\}_{k\ge0}$ are $\Pr$-almost computable, uniformly in $k$.
We now show that, with probability one, $\{\rv D_k\}_{k\ge0}$ is the binary expansion of $\rv X$, thus demonstrating that $\rv X$ is itself a $\Pr$-almost computable random variable.

Let $\rv D$ denote the $\Pr$-almost computable random real whose binary expansion is $\{\rv D_k\}_{k\ge 0}$.
There are two cases to consider.

First, conditioned on the event $\theset {h(\rv N) = \infty}$, we have that $\rv D_k = \rv V_k$ for all $k\ge 0$, and so $\rv D = \rv V = \rv X_{\infty} = \rv X$ almost surely.

In the second case, let $m \in \Nats$, and condition on the event $\theset{ h(\rv N) = m }$.
We must then show that $\rv D = \rv X_m$ a.s.  Note that
\[
\floor{2^{m} \rv X_m} = \floor{2^m \rv V} =
\sum_{k=0}^{m-1} 2^{m-1-k} \rv V_k = \floor{2^m \rv D},
\]
and thus the binary expansions agree for the first $m$ digits.
Finally, notice that $2^{m+1}\rv X_m - 2 \floor{2^{m} \rv X_m} = \rv C + \rv U$, and so
the next binary digit of $\rv X_m$ is $\rv C$, followed
by the binary expansion of $\rv U$, thus agreeing with $\rv D$ for all $k\ge 0$.
\end{proof}

\begin{figure}[t]
\begin{center}
\hspace{-20pt}
\includegraphics[height=.70\linewidth,trim=0 18 0 20]{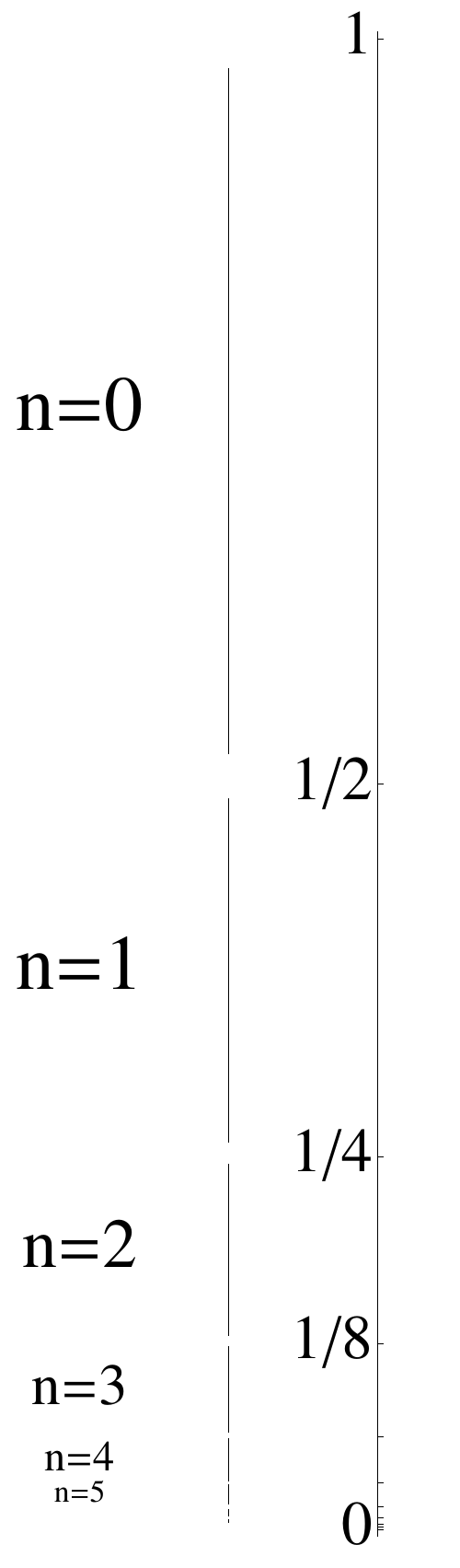}
\includegraphics[width=.70\linewidth]{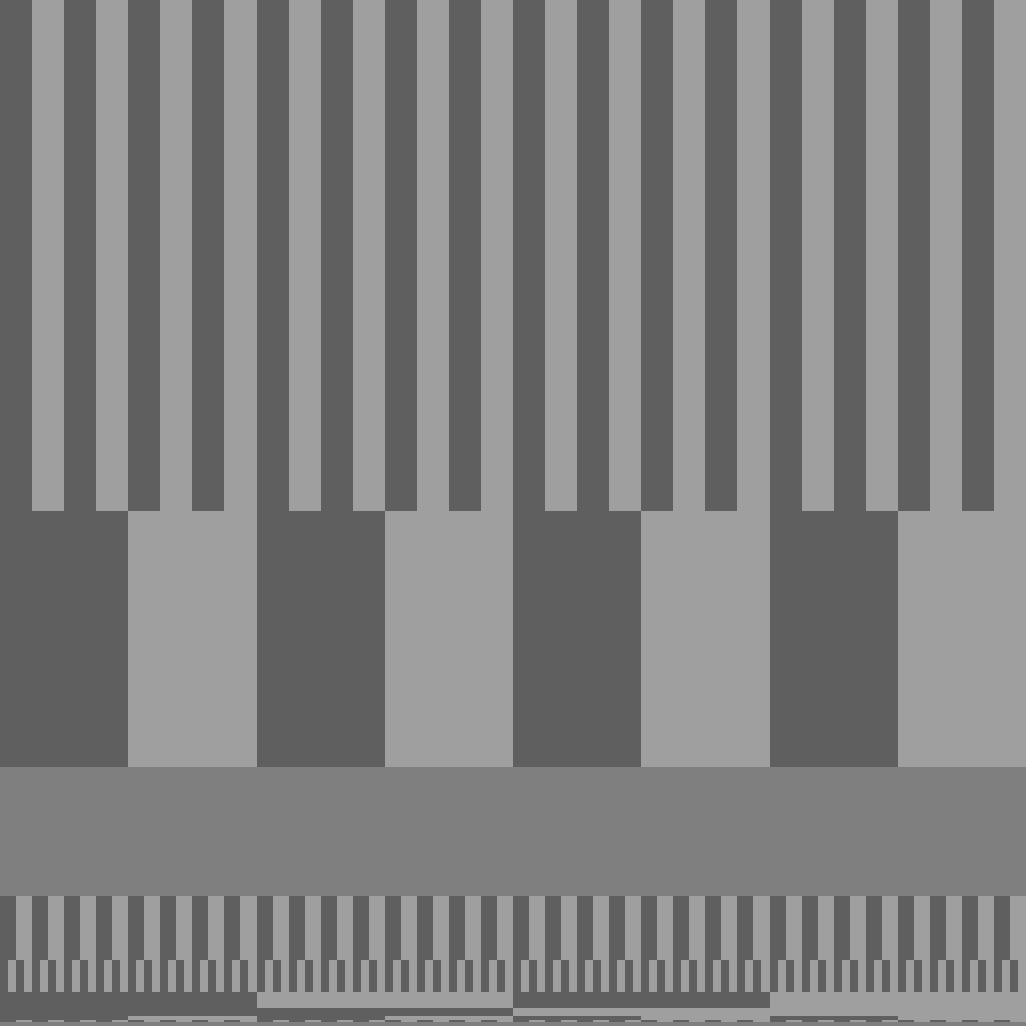}
\end{center}
\caption{
A visualization of the distribution of $(\rv W,\rv Y)$, as defined in
	Theorem~\ref{actualmainresult}.
	In this plot, the darkness of each pixel is proportional to the mass contained in that region. Note that this is not a plot of the (noncomputable) density, but rather of a discretization to a given pixel size.
Regions that appear (at low resolution) to be uniform can suddenly be revealed (at higher resolutions) to be patterned.
Deciding whether the pattern is in fact uniform (as opposed to nonuniform but at a finer granularity than the resolution of this printer/display) is tantamount to solving the halting problem, but it is possible to sample from this distribution nonetheless.
}
\label{haltingdist}
\end{figure}

%%%%%%%%%%%%%%%%%%%%%%

We now show that 
$\CondProbFunc{\rv N}{\rv X}$ is $\Pr_{\rv X}$-almost continuous.
We begin by characterizing the conditional density of $\rv X$ given $\rv N$.

\begin{lemma}
\label{densitylemma}
For each $k\in\Nats \cup \{\infty\}$, the distribution of $\rv X_k$ admits a density $p_{\rv X_k}$
with respect to Lebesgue measure on $[0,1]$ given by
	$p_{\rv X_\infty}(x) = 1$ and
\[
p_{\rv X_k}(x) =
\begin{cases}
\frac 4 3, & \floor{2^{k+1} x} \mathrm{~even}; \\
\frac 2 3, & \floor{2^{k+1} x} \mathrm{~odd},
\end{cases}
\]
	for $k< \infty$. 
\end{lemma}
\begin{proof}
We have $\rv X_\infty = \rv V$ a.s.\ and so the constant function taking the value $1$ is a density of $\rv X_\infty$ with respect to Lebesgue measure on $[0,1]$.

Let $k\in\Naturals$.
With probability one, the integer part of $2^{k+1} \rv X_k$ is $2 \floor{2^k \rv V} + \rv C$ while the fractional part is $\rv U$.
Therefore, the distribution of $2^{k+1}\rv X_k$ (and hence $\rv X_k$) admits a piecewise constant density with respect to Lebesgue measure.

In particular, $\floor{2^{k+1} \rv X_k} \equiv \rv C \imod 2$ almost surely and
$2\floor{2^k\rv V}$ is independent of $\rv C$ and uniformly distributed on $\{0,2,\dotsc,2^{k+1}-2\}$.  Therefore,
\[
\Pr\{ \floor{2^{k+1} \rv X_k}  = \ell\}
= 2^{-k} \cdot
\begin{cases}
\frac 2 3, & \text{$\ell$ even;}\\
\frac 1 3, & \text{$\ell$ odd,}
\end{cases}
\]
for every $\ell \in \{0,1,\dotsc,2^{k+1}-1\}$.
It follows immediately that a density $p$ of $2^{k+1}\rv X_k$ with respect to Lebesgue measure on $[0,2^{k+1}]$ is given by
\[
p(x) =
2^{-k} \cdot
\begin{cases}
\frac 2 3, & \floor{x} \mathrm{~even}; \\
\frac 1 3, & \floor{x} \mathrm{~odd}.
\end{cases}
\]
A density of $\rv X_k$ is then obtained by the rescaling $p_{\rv X_k}(x) = 2^{k+1} \cdot p(2^{k+1}x)$.
\end{proof}

As $\rv X_k$ admits a density with respect to Lebesgue measure on $[0,1]$ for all $k \in \Nats \cup \{\infty\}$, it follows that the conditional distribution of $\rv X$ given $\rv N$ admits a conditional density $p_{\rv X|\rv N}$ (with respect to Lebesgue measure on $[0,1]$)
given by
\[
p_{\rv X|\rv N}(x|n) = p_{\rv X_{h(n)}}(x)
\]
for $x \in [0,1]$ and  $n \in \Nats$.
By Bayes' rule (Lemma~\ref{bayesrule}), 
for every $B \subseteq \Nats$ and $\Pr_{\rv X}$-almost every $x \in [0,1]$,
\[\label{aecontversion}
\CondProbFuncEval{\rv N \in B}{\rv X}{x}
 = \frac 
       {\phi(x,B)} 
       {\phi(x,\Nats)} , 
\]
where
\[ 
\phi (x,B) \defas \sum_{n \in B} \, p_{\rv X | \rv N}(x | n) \cdot \Pr\{\rv N = n \}.
\]

\begin{lemma} \label{lem:effectiveversion}
$\CondProbFunc{\rv N}{ \rv X}$ is $\Pr_{\rv X}$-almost continuous.
\end{lemma}
\begin{proof}
For every $k \in \Nats \cup \theset \infty$,
it holds that $p_{\rv X_{k}}$ is bounded by $4/3$, positive, and continuous on 
the $\Pr_{\rv X}$-measure one set $R$ of nondyadic reals in the unit interval.
Therefore, $p_{\rv X|\rv N}$ is positive, bounded and continuous on $R \times \Nats$, and so, by Proposition~\ref{denstokern}, 
$\CondProbFunc{\rv N}{ \rv X}$ is $\Pr_{\rv X}$-almost continuous.
\end{proof}
%%%%% %%%%%

\begin{lemma} \label{zjcomp}
$\CondProbFunc{\rv N}{ \rv X}$ is $\zj$-computable on a $\Pr_{\rv X}$-measure one set.
\end{lemma}

\begin{proof}
Let $B \subseteq \Nats$ be a computable set.  
By Lemma~\ref{phikappa}, Equation~\eqref{aecontversion}, 
and the computability of division, 
it suffices to show that $\phi_B \defas \phi(\spars,B)$
is $\Pr_{\rv X}$-almost computable from $\zj$. 
Note that $\Pr_{\rv X}$ is absolutely continuous with respect to Lebesgue measure (on $[0,1]$) and vice versa, and so $\Pr_{\rv X}$-almost and Lebesgue-almost computability coincide.  We will therefore drop the measure when referring to almost computability for the remainder of the proof.
Define, for each $n \in \Nats$, the function  $a_n$ on $[0,1]$ given by
\[
\label{an}
a_n(x) \defas
\begin{cases}
3, & h(n) = \infty; \\
2, & h(n) < \infty \text{~and~} \lfloor 2^{h(n)} x \rfloor \text{~even;~and}  \\
4, & h(n) < \infty \text{~and~} \lfloor 2^{h(n)} x \rfloor \text{~odd}.
\end{cases}
\]
Clearly the functions $\{a_n(x)\}_{n\in\Nats}$ are almost computable from $\zj$,
uniformly in $n$. Observe that, for all $n\in\Nats$ and $x \in [0,1]$,
\[
a_n(x) =3\, p_{\rv X | \rv N}(x | n)
\]
by Lemma~\ref{densitylemma}.
Also recall that $\Pr\{\rv N = n \} = \frac 45  \cdot 5^{-n}$
for all $n\in\Nats$.
Hence, for any finite set $F\subseteq \Nats$, the function
\[
\phi_F(x) \defas \phi(x,F) =
\sum_{n\in F} p_{\rv X | \rv N}(x | n) \cdot \Pr\{\rv N = n \} 
=
 \frac4{15} \sum_{n\in F} a_n(x) \cdot 5^{-n}
\]
is almost computable from $\zj$, uniformly in $F$.
However, for every $k\in\Nats$ and 
$x\in[0,1]$,
\[
|\phi_B(x) - \phi_{F_k}(x) | \le 
\sum_{n>k}
p_{\rv X | \rv N}(x | n) \cdot \Pr\{\rv N = n \} 
\le 
 \frac4{15} \cdot 5^{-k},
\]
where $F_{k} = B \cap \{1,\dotsc,k\}$.
It follows that $\phi_B$ is almost computable from $\zj$.
\end{proof}
%%%%% %%%%%

%%%%% %%%%%
\begin{proposition}\label{prevmainnegativeresult}
Let $R \subseteq [0,1]$ be a measurable subset of $\Pr_{\rv X}$-measure
greater than $\frac56$.
For each $k\in\Nats$, the conditional probability map $\CondProbFunc{\rv N=k}{\rv X}$ 
is neither lower nor upper semicomputable on $R$.
\end{proposition}
\begin{proof}
First note that if $R$ had Lebesgue measure no greater than $\frac 34$, then, for $k \in \Nats$,
\[
\Pr_{\rv X_k} (R) \le \frac12\cdot\frac43 + \frac14\cdot \frac23 = \frac56,
\]
and so $\Pr_{\rv X} (R) \le \frac56$, a contradiction.  Hence $R$ has Lebesgue measure greater that $\frac 3 4$.

Fix $k\in\Nats$.
By \eqref{aecontversion} and the definition of $\phi$,
\[\label{kappa-no-P}
\CondProbFuncEval{\rv N=k}{\rv X}{x} =
\frac{p_{\rv X | \rv N}(x | k) \cdot \Pr\{\rv N = k \}}
     {\phi(x,\Nats)}.
\]
for a.e.\ $x$.
The density $p_{\rv X | \rv N}(\cdot | k) = p_{\rv X_{h(k)}}$, 
given by Lemma~\ref{densitylemma},
is a
piecewise constant function and computable on a measure one set.
Furthermore, $\Pr\{\rv N = k \}= 4 \cdot 5^{-k-1}$
is a computable real.
Hence it remains to show that, as a function of $x$, the denominator $\phi(x,\Nats) \defas {\sum_{n \in \Nats} \, p_{\rv X | \rv N}(x | n)
\cdot \Pr\{\rv N = n\}}$ of the right-hand-side 
is neither lower nor upper semicomputable on $R$. We will show the former; the latter follows in a similar fashion.

Recall the sequence of functions $\{a_n\}_{n\in\Nats}$ on $[0,1]$ from \eqref{an},
and define the function $\tau\colon [0,1] \to \Reals$ by
\[
\label{fiveary}
\tau(x) \defas
\sum_{n \in \Nats} a_n(x) \cdot  5^{-n},
\]
which furthermore satisfies
\[
\label{fivearytwo}
\tau(x) 
=\sum_{n \in \Nats} \, 3\, p_{\rv X | \rv N}(x | n) \cdot \frac 5 4 \, \Pr\{\rv N = n\}
\]
for all $x\in[0,1]$.
Observe that for every $x \in [0,1]$, the real $\tau(x)$ has a unique base-5
expansion with all the digits contained in $\{2, 3, 4\}$, which must
necessarily be given by
the sequence
$a_0(x), a_1(x), \dotsc$.
For each $\ell \in \Nats$, define
\[
\tau_{\ell}(x)
      \defas 5^{\ell} \Bigl(\tau(x) - \sum_{n < \ell} a_n(x) \cdot 5^{-n} \Bigr)
      = \sum_{n \in \Nats} a_{n+\ell}(x) \cdot 5^{-n}.
\]
Note that 
$\tau_{0}$ ($= \tau$) is lower semicomputable on $R$ precisely when 
$\phi(x,\Nats)$
is lower semicomputable on $R$.
Further note that if $a_{\ell}(x) = 2$, then $\tau_{\ell}(x) < 3$
(where the inequality is strict because for every $n$, there exists an $m > n$ such that $h(m) = \infty$).
On the other hand, if $a_{\ell}(x) \ge 3$,
then $\tau_{\ell}(x) > 3$.

We now prove by induction,
uniformly in $\ell$, that, if $\tau_\ell$ is lower
semicomputable on $R$, 
then
$\tau_{\ell+1}$ is lower semicomputable on $R$,
and we can compute whether or not $h(\ell)$ is finite.
This then implies that if $\tau$ were lower semicomputable
on $R$, then $\zj$ would be computable, a contradiction.

Suppose
$\tau_{\ell}$ is lower semicomputable on $R$.  
Then we can compute (as
a c.e.\ open subset of $[0,1]$) a set $S$ such that $S \cap R = \tau_\ell^{-1}[(3,\infty)] \cap R$.
Simultaneously run $M_{\ell}$ (on input $0$).
If $h(\ell) < \infty$, then we will eventually notice this fact by
observing that $M_{\ell}$ halts. It is also the case that
\linebreak
$S \cap R \subseteq \{x \in [0,1] \st  a_{\ell}(x) = 4\}$,
and so, by the definition of the functions $\theset{a_n}$, the set $S \cap R$ has Lebesgue measure at most $\frac12$. Therefore $S$ has measure 
at most $\frac12$ plus
the measure of $[0,1] - R$,
i.e., $S$ has measure less than $\frac12 + (1 - \frac34) = \frac34$.
On the other hand, if $h(\ell) = \infty$, then
$S\supseteq R$, and hence $S$
has Lebesgue measure greater than $\frac34$, which we will eventually notice (and
which rules out the first case).
Hence we can compute whether or not $h(\ell)$ is finite, which in turn
implies that
$a_{\ell}$ is computable on a measure one set.
Therefore, as 
$\tau_{\ell+1}(x) = 5 \cdot \left(\tau_{\ell}(x) -
a_{\ell}(x)\right)$, the function
$\tau_{\ell+1}$ is lower semicomputable on $R$.
\end{proof}

We may summarize our central result as follows.
\begin{theorem}
	\label{actualmainresult}
	There are $\Pr$-almost computable random variables $\rv W$ and $\rv Y$ on $[0,1]$ such that the
conditional distribution map $\CondProbFunc{\rv Y}{\rv W}$ is
	$\Pr_{\rv W}$-almost continuous  but
	not $\Pr_{\rv W}$-almost computable.
\end{theorem}
\begin{proof}
	Let $\rv X$ and $\rv N$ be as above, 
	let $\rv Y$ be uniformly distributed on $[0,1]$, and let
	$\rv M$ be the 
	geometric($\frac12$) random variable given by
$\rv M \defas  \lfloor - \log_2 \rv Y \rfloor$. 
Finally, let $\rv W$ be such that
$\CondProb{\rv W} {\rv Y}(\varpi) = \CondProbFuncEval{\rv X}{\rv N}{\rv M(\varpi)}$.
         Note that we have
         $\CondProbFuncEval{\rv Y \in (2^{-n-1},2^{-n})}{\rv W}{x}
          = \CondProbFuncEval{\rv Y \in [2^{-n-1},2^{-n}]}{\rv W}{x}$
\linebreak
$ = \CondProbFuncEval{\rv N = n}{\rv X}{x}$ for $n\in\Nats$ and $x\in[0,1]$.
	The result then follows from 
Lemma~\ref{lem:effectiveversion} and Proposition~\ref{prevmainnegativeresult}. 
\end{proof}

For a visualization of the distribution of $(\rv W, \rv Y)$, see Figure~\ref{haltingdist}.

%%%%%% %%%%%%

The proof of Proposition~\ref{prevmainnegativeresult}
shows that not only is
the conditional distribution map $\CondProbFunc{\rv N}{\rv X}$ not $\Pr_{\rv X}$-almost computable,
but that, in fact,
it computes the halting set $\zj$;
to make this precise, we would define the notion of an oracle that encodes $\CondProbFunc{\rv N}{\rv X}$ using,
e.g., infinite strings as in the
Type-2 Theory of Effectivity.
Despite not having a
definition of computability \emph{from} the conditional distribution map,
we can easily relativize the notion of computability \emph{for} the
conditional distribution map,
to obtain
the following.
\begin{corollary}
If $\CondProbFunc{\rv N}{\rv X}$ is $A$-computable on a set of $\Pr_{\rv X}$-measure greater than $\frac 5 6$
for an oracle $A\subseteq
	\Naturals$, then $A$ computes the halting set, i.e., $A \ge_{\mathrm{T}} \zj$.
In particular, $\CondProbFunc{\rv N}{\rv X}$ is not computable on any set of $\Pr_{\rv X}$-measure greater than $\frac 5 6$, and hence no version of $\CondProbFunc{\rv N}{\rv X}$ is $\Pr_{\rv X}$-almost computable.
\label{cor-lowerbound}
\end{corollary}
\begin{proof}
Suppose $A$ is such that $\CondProbFunc{\rv N}{\rv X}$ is $A$-computable on a 
set of $\Pr_{\rv X}$-measure greater than $5/6$.
Then
$\CondProbFunc{\rv N=1}{\rv X}$  is $A$-computable on the same set.
But then, by the argument in the proof of 
Proposition~\ref{prevmainnegativeresult},  the function $h$, and hence the
halting set $\zj$, is computable
from $A$.
\end{proof}

On the other hand, by 
Lemma~\ref{zjcomp}, 
the conditional distribution map $\CondProbFunc{\rv N}{\rv X}$ is
in fact $\zj$-computable on a $\Pr_{\rv X}$-measure one set, and so the bound in Corollary~\ref{cor-lowerbound} is
the best possible.

Computable operations map computable points to computable points,
and so 
Corollary~\ref{cor-lowerbound}
provides another context in which conditioning operators are noncomputable (cf.\ Proposition~\ref{mainnegativeresult}).

The next result shows that Proposition~\ref{prevmainnegativeresult} also rules out the 
computability of conditional probability maps in the weaker sense of $L^1(\Pr_{\rv X})$-computability.  
This extends
\citep[][Prop.~3]{HR11}, 
which states that there is a pair of computable measures $\mu \ll \nu$ on a computable Polish space 
such that the Radon--Nikodym derivative $\dee \mu / \dee \nu$ is not ${L^1\textrm{-computable}}$.  
Namely, we show that in the case of measures on
$[0,1]$, the measures can be taken to be of the form
$\mu = \Pr \{ \rv X \in \pars, \rv N = k\}$ 
and $\nu = \Pr\{\rv X \in \pars \} = \Pr_{\rv X}$.

\begin{proposition}\label{nonl1comp}
For every $k\in\Nats$,
the map $\CondProbFunc{\rv N = k}{\rv X}$
is not 
$L^1(\Pr_{\rv X})$-computable.
\end{proposition}
\begin{proof}
Let $k\in\Nats$.
By Proposition~\ref{prevmainnegativeresult}, 
the conditional probability map $\CondProbFunc{\rv N = k}{\rv X}$
is not 
computable on any 
measurable
set $R$
of $\Pr_{\rv X}$-measure greater than $5/6$.
On the other hand, 
by Lemma~\ref{layer},
the conditional probability map is $L^1(\Pr_{\rv X})$-computable only if,
for each $r \in \Nats$, the map
is computable on some set of $\Pr_{\rv X}$-measure at least $1-2^{-r}$, uniformly in $r$.
This does not hold, and so every conditional probability map $\CondProbFunc{\rv N = k}{\rv X}$ 
is not $L^1(\Pr_{\rv X})$-computable.
\end{proof}

It is natural to ask whether this construction can be modified to
produce a pair of $\Pr$-almost computable random variables, like $\rv N$ and $\rv X$, such that the corresponding conditional
distribution map is not $\Pr_{\rv X}$-almost computable even though it has an
\emph{everywhere continuous} version.
We provide such a strengthening
in the next section.

%%%%%%%%%%%%%%%%%%

\section{Noncomputable Everywhere Continuous Conditional Distributions}
\label{Sec:c.e. neg cont}

In Section~\ref{Sec:c.e. neg}, we demonstrated a pair of  
computable random variables 
that admit a conditional distribution 
that is continuous on a measure one set but still noncomputable on every measure one set.
It is therefore natural to ask whether we can
construct a pair of random variables $(\rv Z, \rv N)$ that is
computable and admits an \emph{everywhere} continuous version of the
conditional distribution map $\CondProbFunc{\rv N }{ \rv Z}$,
which is itself
nonetheless not computable.
In fact, we do so now, using a construction similar to that of $(\rv X, \rv N)$ in Section~\ref{Sec:c.e. neg}.

If we think of the construction of the $k$th bit of $\rv X$ as an iterative process, we see that there are two distinct stages. During the first stage, which occurs so long as $k < h(\rv N)$, the bits of $\rv X$ simply mimic those of the uniform random variable $\rv V$. Then during the second stage, once $k \geq h(\rv N)$, the bits mimic that of $\frac 1 2 (\rv C + \rv U)$.

Our construction of $\rv Z$ will differ in the second stage, where the bits of $\rv Z$ will instead mimic those of a random variable $\rv S$ specially designed  to smooth out the rough edges caused by the biased Bernoulli random variable $\rv C$, while still allowing us to encode the halting set.
In particular, $\rv S$ will be absolutely continuous and will have an infinitely differentiable density.

We now begin the construction.  Let $\rv N$, $\rv U$, $\rv V$, and $\rv C$ be as in the first construction.
We next define several random variables from which we will construct $\rv S$, and then $\rv Z$.

\begin{lemma}
\label{c-infinity-step}
There is a random variable $\rv F$ in $[0,1]$ with the following properties:
\begin{enumerate}
\item $\rv F$ is $\Pr$-almost computable.
\item $\Pr_{\rv F}$ admits a computable
density $p_{\rv F}$ with respect to Lebesgue
measure (on $[0,1])$ that is infinitely differentiable
everywhere.
\item $p_{\rv F}(0) = \frac 2 3$ and $p_{\rv F}(1) = \frac 4 3$.
\item $\frac{d^n_+}{dx^n} p_{\rv F}(0) = \frac{d^n_-}{dx^n} p_{\rv F}(1) = 0$, for all $n\ge 1$ (where $\frac{d^n_-}{dx^n}$ and $\frac{d^n_+}{dx^n}$ are the left and right derivatives respectively).
\qed
\end{enumerate}
\end{lemma}
(See Figure~\ref{stepdensity} for one such random variable.)
Let $\rv F$ be as in Lemma~\ref{c-infinity-step},
and independent of all earlier random variables mentioned.
Note that $\rv F$ is almost surely nondyadic and so the $r$-th bit $\rv F_r$  of $\rv F$ is a $\Pr$-almost computable random variable, uniformly in $r$.

%%%%%%%%%%%%%%%%%%%%%
%%%%%%%%%%%%%%%%%%%%%
\newcommand{\bumpint}{\Phi}

\begin{figure}[t]
\centering
\includegraphics[height=5.05cm]{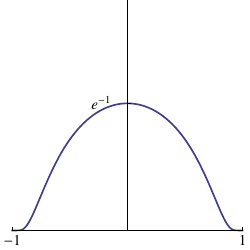}
\ \ \ \
\includegraphics[height=5.05cm]{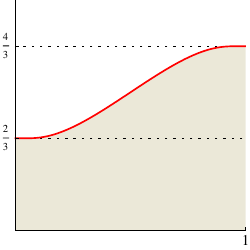}
\caption{
(left)\ \  The graph of the function defined by 
$f(x) = \exp\{- (1-x^2)^{-1} \}$, 
for $x \in (-1,1)$, and $0$ otherwise, a $C^\infty$ bump function whose derivatives at $\pm 1$ are all 0.
\ \ \ (right)\ \  A density 
$p(y) = 2/3 \, ( 1 + {\bumpint(2y-1)} / {\bumpint(1)} )$, 
for $y \in (0,1)$, of a random variable satisfying Lemma~\ref{c-infinity-step}, where $\bumpint(y) = \int_{-1}^y f(x)\,\dee x$ is the integral of the bump function.
}
\label{stepdensity}
\end{figure}
%%%%%%%%%%%%%%%%%%%%%
%%%%%%%%%%%%%%%%%%%%%

Let $\rv D$ be a $\Pr$-almost computable random variable,
independent of all earlier random variables mentioned,
and uniformly distributed on $\{0,1, \dotsc, 7 \}$. Consider
\[
\rv S = \frac 1 8 \times
\begin{cases}
\rv F, & \text{if } \rv D = 0; \\
4 + (1- \rv F), & \text{if } \rv D = 4; \\
4 \rv C + (\rv D \bmod 4) + \rv U, & \text{otherwise.}
\end{cases}
\]
It is clear that $\rv S$ is also $\Pr$-almost computable, and straightforward to show that
\begin{itemize}
\item[\emph{(i)}] $\Pr_{\rv S}$ admits an infinitely differentiable 
and computable
density $p_{\rv S}$ with respect to Lebesgue measure on $[0,1]$; and
\item[\emph{(ii)}] For all $n \ge 0$, we have $\frac{d^n_+}{dx^n}p_{\rv S}(0)=
\frac{d^n_-}{dx^n}p_{\rv S}(1)$.
\end{itemize}

(For a visualization of the density $p_{\rv S}$ see
Figure~\ref{smooth-dens}.)

Next we define, for every $k \in \Nats$, the random variables $\rv Z_k$ mimicking the construction of $\rv X_k$. Specifically, for $k \in \Nats$, define
\[
\rv Z_k \defas \frac{\lfloor 2^k \rv V \rfloor + \rv S}{2^{k}},
\]
and let $\rv Z_\infty \defas \lim_{k \rightarrow \infty} \rv Z_k = \rv V $ a.s.\ \
Then the $n$th bit of $\rv Z_k$ is
\[
(\rv Z_{k})_n =
\begin{cases}	
\rv V_n, & n < k;\\
\rv S_{n-k}, & n \geq k
\end{cases} \qquad \textrm{a.s.},
\]
where $S_\ell$ denotes the $\ell$th bit of $S$.
It is 
straightforward to show from \emph{(i)} and \emph{(ii)} above that
$\Pr_{\rv Z_k}$ admits an infinitely differentiable 
density $p_{\rv Z_k}$ with respect to Lebesgue measure on $[0,1]$.

To complete the construction, we define $\rv Z \defas \rv Z_{h(\rv N)}$.  The following results are analogous to those in the almost continuous construction.

\begin{lemma}
The random variable $\rv Z$ is $\Pr$-almost computable.
	\qed
\end{lemma}
\begin{lemma}
There is an everywhere continuous version of $\CondProbFunc{\rv N}{\rv Z}$.
\end{lemma}
\begin{proof}
By construction, the conditional density of $\rv Z$ given $\rv N$ is everywhere continuous, bounded, and positive.  The result follows from Proposition~\ref{denstokern} for $R=[0,1]$.
\end{proof}
%%%%%%%%%%%%%%%%%%%%%%%%%%

We next show that, 
for each $k\in\Nats$, the conditional probability map $\CondProbFunc{\rv N=k}{\rv Z}$
is not computable.  Our proof relies on the fact that, 
for each $k\in\Nats$,
there is a large set of 
points $x\in[0,1]$ for which the density $p_{\rv Z | \rv N}(x | k) \defas
p_{\rv Z_{h(k)}}(x)$
agrees with 
$p_{\rv X | \rv N}(x | k)$. This will then allow us to use techniques
similar to those of Proposition~\ref{prevmainnegativeresult}.

We say a real $x \in [0,1]$ is \defn{valid for $\Pr_{\rv S}$}
if $x \in (\frac 1 8, \frac 1 2) \cup (\frac 5 8, 1)$.  In particular, when $\rv D \not\in \{0,4\}$, then $\rv S$ is valid for $\Pr_{\rv S}$.
The following are then 
consequences of the construction of $\rv S$ and the definition of valid points:
\begin{itemize}

\item[\emph{(iii)}] If $x$ is valid for $\Pr_{\rv S}$ then $p_{\rv S}(x) \in \{\frac 2 3, \frac 4 3\}$.
In particular, 
$p_{\rv S}(x) = \frac 43$ for
$x \in (\frac 18, \frac 12)$, and
$p_{\rv S}(x) = \frac 23$ for
$x\in (\frac 58, 1)$.

\item[\emph{(iv)}] The Lebesgue measure
of points valid for $\Pr_{\rv S}$ 
is $(\frac12 - \frac18) + (1 - \frac58) = \frac34$.

\end{itemize}

For $k < \infty$, we say that $x \in [0,1]$ is \defn{valid for $\Pr_{\rv
Z_{k}}$} if the fractional part of $2^k x$ is valid for $\Pr_{\rv S}$, and
we say that $x$ is \defn{valid for $\Pr_{\rv Z_\infty}$}, for all $x$.  Let
$A_k$ be the collection of $x$ valid for $\Pr_{\rv Z_k}$, and let $A_\infty = [0,1]$.  
Note that, at all points $x$ that are valid for $\Pr_{\rv Z_k}$,
the density $p_{\rv Z_k}(x)$ agrees with $p_{\rv X_k}(x)$, and, at all
points $x$ that are valid for $\Pr_{\rv Z_\infty}$,
the density $p_{\rv Z_\infty}(x)$ agrees with $p_{\rv X_\infty}(x)$.

For each $k< \infty$, the set $A_k$ consists of 
the set of points that are valid for $\Pr_{\rv S}$
first dilated by a factor of $2^{-k}$ and then tiled $2^k$-many
times, and so $A_k$ has Lebesgue measure $\frac34$; the
set $A_\infty = [0,1]$ has Lebesgue measure one.

%%%%%%%%%%%%%%%%%%%%%%%%
%%%%%%%%%%%%%%%%%%%%%%%%
\begin{figure}[t]
\centering
\includegraphics[height=5.05cm]{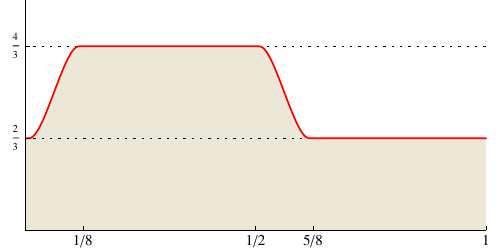} \\
\caption{
Graph of $p_{\rv S}$, the density of $\rv S$, when $\rv S$ is constructed from $\rv F$ as given in Figure~\ref{stepdensity}.
}
\label{smooth-dens}
\end{figure}
%%%%%%%%%%%%%%%%%%%%%%%%
%%%%%%%%%%%%%%%%%%%%%%%%

%%%%%%%%%%%%%%%%%%%%%%%%%%
\begin{proposition}
Let $R \subseteq [0,1]$ be a measurable subset of $\Pr_{\rv Z}$-measure greater than $11/12$.
Then for each $k\in\Nats$, the conditional probability map $\CondProbFunc{\rv N=k}{\rv Z}$
is neither lower semicomputable nor upper semicomputable on $R$.
\end{proposition}
%%%%%%%%%%%%%%%%%%%%%%%%%%
\begin{proof}
We proceed analogously to the proof of
Proposition~\ref{prevmainnegativeresult}.
First note that if $R$ had Lebesgue measure no greater than $\frac 78$, then, for $k \in \Nats$,
\[
\Pr_{\rv Z_k} (R) \le \frac12\cdot\frac43 + \frac38\cdot \frac23 = \frac{11}{12}
\]
and so $\Pr_{\rv Z} (R) \le \frac{11}{12}$, a contradiction.  Hence $R$ has Lebesgue measure greater that $\frac 7 8$.

Fix $k\in\Nats$.
By Bayes' rule (Lemma~\ref{bayesrule}),
\[
\label{bayesruleeq}
\CondProbFuncEval{\rv N=k}{\rv Z}{x} =
\frac{p_{\rv Z | \rv N}(x | k) \cdot \Pr\{\rv N = k \}}
     {\sum_{n \in \Nats} \, p_{\rv Z | \rv N}(x | n) \cdot \Pr\{\rv N = n\}}.
\]
for $\Pr_{\rv Z}$-a.e.\ $x$, and  hence for Lebesgue-a.e.\ $x$.

Now,
$p_{\rv Z | \rv N}(\cdot | k)$ is
piecewise
(on a finite number of pieces with rational endpoints)
the computable dilation and translation of the computable function $p_{\rv
S}$.  By \emph{(ii)} above, the endpoints of the pieces line up, and so
$p_{\rv Z | \rv N}(\cdot | k)$ is computable.

Because $\Pr\{\rv N = k \} = \frac 45 \cdot 5^{-k}$ is a computable real,
it remains to show that, as a function of $x$, the denominator 
${\sum_{n \in \Nats} \, p_{\rv Z | \rv N}(x | n)
\cdot \Pr\{\rv N = n\}}$ 
of the right-hand side of \eqref{bayesruleeq}
is neither lower nor upper semicomputable on $R$. We will show the former; the latter follows in a similar fashion.
Define, for each $\ell\in\Nats$ the function $\psi_\ell\colon [0,1] \to \Reals$ by
\[
\psi_\ell(x) 
=5^\ell
\sum_{n \ge \ell} 
\, 3\, p_{\rv Z | \rv N}(x | n) \cdot \frac 5 4 \, \Pr\{\rv N = n\}.
\]
Note that $\psi_0$ is equal to $\frac{15}4$ times the denominator.
Hence it suffices to show that $\psi_0$ is not lower semicomputable on $R$.

We now prove by induction,
uniformly in $\ell$, that, if $\psi_\ell$ is lower
semicomputable on $R$, 
then
$\psi_{\ell+1}$ is lower semicomputable on $R$,
and we can compute whether or not $h(\ell)$ is finite.
This then implies that if $\psi_0$ were lower semicomputable
on $R$, then $\zj$ would be computable, a contradiction.

Suppose that $\psi_\ell$ is lower semicomputable
on $R$.
Then we can compute (as a c.e.\ open subset of $[0,1]$) a set $S$ such that
$S \cap R = \psi_\ell^{-1}[(3,\infty)] \cap R$.

If $\ell$ is such that $h(\ell) = \infty$, then
for $x\in[0,1]$ we have
$
p_{\rv Z | \rv N}(x \,|\, \ell)  = 1
$,
and so
\[
\psi_\ell(x) 
&= 3 + \frac 1 5\, \psi_{\ell+1}(x) > 3.
\]
In particular, $S$ will contain $R$, and hence have 
Lebesgue measure
greater than 
$\frac78$.

On the other hand, if $\ell$ is instead such that $h(\ell) < \infty$, then
$S\cap R \cap A_{h(\ell)}$
has Lebesgue measure at most $\frac12$. Hence 
\[
S \subseteq (S\cap R \cap A_{h(\ell)}) 
\,\cup \,
([0,1] \setminus R) 
\,\cup \,
([0,1] \setminus A_{h(\ell)}),
\]
and so the 
Lebesgue measure
of $S$ is less than $\frac 12 + (1-\frac 78) + (1-\frac34) = \frac78$.

Simultaneously (1) enumerate a c.e.\ sequence of basic open sets whose union
is $S$, hence computing arbitrarily good lower bounds on
its measure, and (2) run $M_{\ell}$ (on input $0$).
Either $S$ will have measure greater than $\frac78$
or $M_{\ell}$ will halt,
but not both. Hence we will eventually learn whether 
or not $h(\ell)$ is finite, which in turn shows that
$p_{\rv Z | \rv N}(\cdot \,|\, \ell)$
is
computable on $[0,1]$,
and so $\psi_{\ell+1}$ is also lower
semicomputable on $R$.
\end{proof}

%%%%%% %%%%%%

As before, it follows immediately that $\CondProbFunc{\rv N}{\rv Z}$ is not computable, even on a $\Pr_{\rv Z}$-measure one set.
It is possible to carry on the same development, showing that  
the conditional probability map
is not computable
as an element in
$L^1(\Pr_{\rv Z})$.
We state only the following strengthening of
Corollary~\ref{cor-lowerbound}.

\begin{corollary} \label{newmainnegativeresult}
Let $\Phi$ be a conditioning operator for the set of probability distributions on pairs $(\rv X,\rv Y)$ of 
random variables in $[0,1]$ such that there exists 
an everywhere continuous version of the conditional distribution map $\CondProbFunc{\rv Y}{\rv X}$. 
Then $\Phi$ is noncomputable.
\qed
\end{corollary}

%%%%%%%%%%%%%%%%%%%%%%

%%%%%%%%%%%%%%%%%%%%%%%%%%%%%%%%%%%%%%
%%%%%%%%%%%%%%%%%%%%%%%%%%%%%%%%%%%%%%
%%%%%%%%%%%%%%%%%%%%%%%%%%%%%%%%%%%%%%

\section{Positive Results}
\label{positiveresults}

Despite the fact that conditioning is not computable in general, 
many practitioners view conditioning as a routine operation involving the application of Bayes' rule.
Indeed, below we show that
conditioning is computable under mild assumptions when the observation is discrete, or when Bayes' rule is applicable, or when the observation is corrupted by independent ``smooth'' computable noise.  We conclude by examining the setting where a notion of symmetry known as \emph{exchangeability} holds.

\subsection{Discrete Random Variables}
\label{discreterandomvariablessection}

We begin with the problem of conditioning on a random variable that takes
values in a discrete set.
Given an appropriate notion of computability for discrete sets, 
conditioning is always possible in this setting, as it reduces to the elementary notion of conditional probability with respect to single events.

\begin{definition}[Computably discrete set] \label{compdisc2}
Let $S$ be a computable Polish space.
A subset $D \subseteq S$ is \defn{computably discrete} 
when 
there exists a function $f \colon S \to \Nats$ that is computable and injective
on $D$. In particular, such a $D$ is countable.
We call $f$ the \defn{witness} to the discreteness of $D$. 
\end{definition}

%%%%%%%%%%%%%%%%%%%%%%%%%%%%%%%%%%%%%%
\begin{proposition}\label{discreteconditioning}
Let $\rv X$ and $\rv Y$ be $\Pr$-almost computable random variables in computable Polish spaces $S$ and $T$, respectively, 
let $D \subseteq S$ be a computably discrete subset with witness $f$, 
and assume that $\Pr \{\rv X = d \} > 0$ for all $d \in D$.  
Then the conditional distribution map $\CondProbFunc{\rv Y }{ \rv X}$ is computable on $D$, uniformly in $\rv X$, $\rv Y$, and $f$.
\end{proposition}
%%%%%%%%%%%%%%%%%%%%%%%%%%%%%%%%%%%%%%
\begin{proof}
Let $A \subseteq T$ be a $\Pr_{\rv Y}$-almost decidable set.
	By Lemma~\ref{phikappa}, it suffices to show that \linebreak $\CondProbFunc{\rv Y \in A}{ \rv X}$ is computable on $D$, uniformly in 
	$\mu$, $f$, and
	(the witness for the $\Pr_{\rv Y}$-almost decidability of) $A$.
Let $x\in D$.
Uniformly in $f$ and $x$, 
we can find a witness 
to a $\Pr_{\rv X}$-almost decidable 
set $B_x \subseteq S$ such that 
$x \in B_x$  and no other element of $D$
is in $B_x$.
It then follows that
\[
\CondProbFuncEval{\rv Y \in A}{ \rv X}{x}
= 
\CondProb{\rv Y \in A}{\rv X\in B_x}
\ \ \as
\]
By Lemma~\ref{condcontset}, 
$\CondProb{\rv Y \in A}{\rv X\in B_x}$ is a computable real 
 uniformly in $\mu$ and
	(the witness for the $\Pr_{\rv X}$-almost decidability of) $B_x$,
 hence uniformly in $\mu$, $f$, and $x$.
\end{proof}

The proof above relies on the ability to compute probabilities for continuity sets. 
In practice, computing probabilities can be inefficient. We give an alternative proof of Proposition~\ref{discreteconditioning}
via a rejection-sampling argument, which yields an algorithm that can be much more efficient.
We begin with a version of a well-known result.

\begin{lemma}[rejection sampling]
\label{rejection}
Let $\rv X$ and $\rv Y$ be random variables in computable Polish spaces $S$ and $T$, respectively, 
let $B \subseteq S$ be a measurable set with positive $\Pr_{\rv X}$-measure,
let $(\rv X_0,\rv Y_0),(\rv X_1,\rv Y_1),\dotsc$ be a sequence of i.i.d.\ copies of $(\rv X,\rv Y)$,
and define $\nu$ to be the distribution of $(\rv X_0,\rv X_1,\dotsc)$.
Further define $g\colon S^\Nats \to \Nats$ to be the map
	\[(x_0, x_1, \dots) \mapsto \inf \, \{ n \in \Nats \st  x_n \in B \} \]
and set
\[ \zeta \defas g(\rv X_0,\rv X_1,\dots). \]
Then $\rv Y_\zeta$ has distribution $\Pr \{ \rv Y \in \pars \given \rv X \in B \}$
and, if $B$ is a $\Pr_{\rv X}$-almost decidable set, 
then $g$ is $\nu$-almost computable, uniformly in a witness for $B$.
\end{lemma}
\begin{proof}
	For all $n \in \Nats$, we have $\Pr \{\zeta = n \} = \Pr\{ \rv X_n \in B \} \prod_{i=0}^{n-1} \Pr\{ \rv X_j \not\in B \}  
	= {\Pr \{\rv X \in B \}}\, (1 - \Pr \{\rv X \in B \})^n$. 
Summing this expression over all $n\in\Nats$, we obtain	
\[\label{bcfact}
\Pr \{\zeta < \infty \} = 1.
\]
For $\Pr_{\zeta}$-almost all $n \in \Nats$,
\[
\Pr\{\rv Y_{\zeta} \in \pars |\, \zeta = n \}
	&= \Pr\{\rv Y_n \in \pars |\, \rv X_n \in B \text{ and } \rv X_j \not\in B \text{ for all $j < n$} \}  \\
&= \Pr\{\rv Y_n \in \pars |\, \rv X_n \in B \}  
= \Pr\{\rv Y \in \pars |\, \rv X \in B \},
\]
where the first equality follows from the definition of $\zeta$, the second from the i.i.d.\ property, and the last by definition.
As the right hand side has no dependence on $n$, it follows from the chain rule of conditional
expectation and \eqref{bcfact} that
$\Pr\{\rv Y_{\zeta} \in \pars  \} = \Pr\{\rv Y \in \pars |\, \rv X \in B \}$.

We now establish the $\nu$-almost computability of $g$. 
It follows from $\Pr\{\zeta < \infty \} =1$ that $g$ is $\Nats$-valued on a $\nu$-measure one set.
By definition, the characteristic function for $B$, written \linebreak $\chi_B \colon S \to \{0,1\}$, is $\Pr_{\rv X}$-almost computable uniformly in a witness for the almost decidability of $B$.
Hence $g(x_0,x_1,\dotsc) = \min \{ n \in \Nats \st \chi_B(x_n) = 1 \}$ 
on a $\nu$-measure one set $A\subseteq S^\Nats$ of input sequences.  
But then $A \cap g^{-1}(n) = A \cap \bigl( (\chi_B^{-1}(0))^{n-1} \times \chi_B^{-1}(1) \times S^\Nats \bigr)$,
and so $g$ is $\nu$-almost computable.
\end{proof}

We may now give an alternative proof of Proposition~\ref{discreteconditioning}.

\begin{proof}[Proof of Proposition~\ref{discreteconditioning} (via rejection sampling)]
Let $\mu$ be the distribution of $(\rv X, \rv Y)$, which is computable. 
Uniformly in $\mu$, we may compute a sequence $(\rv X_0,\rv Y_0), (\rv X_1,\rv Y_1), \dots$ 
of independent $\Pr$-almost computable random variables with distribution $\mu$. 
Let $f$ be a witness to the computable discreteness of $D$. 
Then $(f(\rv X_0),\rv Y_0), (f(\rv X_1),\rv Y_1), \dots$ is also a computable sequence of independent $\Pr$-almost computable random variables.
Uniformly in $f$, we can compute a sequence of disjoint $\Pr_{\rv X}$-continuity sets $\<B_k\>_{k \in \Nats}$ such that $x \in B_{f(x)}$ for $\Pr_{\rv X}$-almost all $x$. 
For $k \in \Nats$, let $g_k$ be the function from Lemma~\ref{rejection} with respect to the set $B_k$. 
Uniformly in $k$, the random variable $\zeta_k = g_k(f(\rv X_0),f(\rv X_1),\dots)$ is $\Pr$-almost computable,
and then so is the random variable $\rv Y_{\zeta_k}$.
For $k \in \Nats$, let $\gamma_k$ be the distribution of $\rv Y_{\zeta_k}$.
By Lemma~\ref{rejection}, $\gamma_k = \Pr \{ \rv Y \in \pars \given f(\rv X) = k \}$ and $\gamma_k$ is computable, uniformly in $k$, $f$, and $\mu$.
It follows that the map $\gamma \colon \Nats \to \ProbMeasures(T)$ given by $k \mapsto \gamma_k$ is computable, uniformly in $f$ and $\mu$.
Note, however, that for $\Pr_{\rv X}$-almost all $x$, we have 
$\CondProbFuncEval{\rv Y}{\rv X}{x} = \Pr \{ \rv Y \in \pars \given f(\rv X) = f(x) \}$.
Therefore $\CondProbFunc{\rv Y}{\rv X} = \gamma \circ f$ is $\Pr_{\rv X}$-almost computable, uniformly in $f$ and $\mu$.
\end{proof}

Rejection sampling has been used to give informal semantics of conditioning
operators for discrete random variables in probabilistic programming languages 
such as $\lambda_\circ$ \citep{1452048} and Church \citep{GooManRoyBonTen2008}.

%%%%%%%%%%%%%%%%%%%%%%%%

\subsection{Continuous and Dominated Setting}
\label{cds}

The most common way
to to perform conditioning given
a continuous random variable is via Bayes' rule, which requires the existence of a conditional density.  
Using the characterization of computable measures in terms of the computability of integration,
we can characterize when Bayes' rule is computable.

\begin{proposition}\label{denstokerncomp}
Let $\rv X$ and $\rv Y$ be $\Pr$-almost computable random variables on computable Polish spaces $S$ and $T$, and let $R \subseteq S$.
If $p_{\rv X | \rv Y}(x | y)$ is a conditional density of $\rv X$ given $\rv Y$ that 
is positive, bounded, and computable on $R \times T$, 
then $\kappa$ as defined in \eqref{formalbayes} is computable on $R$. 
In particular, if $R$ is a $\Pr_{\rv X}$-measure one subset, then $\kappa$ is a $\Pr_{\rv X}$-almost computable version.
\end{proposition}

\begin{proof}
By Bayes' rule (Lemma~\ref{bayesrule}), a version of $\CondProbFunc{\rv Y}{\rv X}$ is given by the ratio 
in Equation~\eqref{formalbayes}:
\[
\label{aoeuformalbayes}
\kappa(x,B)
=
\frac{\int_{B} p_{\rv X | \rv Y}(x | y) \, \Pr_{\rv Y}(\dee y)}
     {\int_{\phantom{B}}   p_{\rv X | \rv Y}(x | y) \, \Pr_{\rv Y}(\dee y)}, \qquad B \in \Borel_T.
\]
By Proposition~\ref{inttwovartoonever}, when $B$ is an $\Pr_{\rv Y}$-almost decidable set, $x \mapsto \int_{B} p_{\rv X | \rv Y}(x | y) \, \Pr_{\rv Y}(\dee y)$ is a computable function of $x$ on $R$, uniformly in a witness for the $\Pr_{\rv Y}$-almost decidability of $B$. 
Further we know that $\int p_{\rv X | \rv Y}(x | y) \, \Pr_{\rv Y}(\dee y)$ is non-zero on a $\Pr_{\rv X}$-measure one set,
	and so the ratio $\kappa(\cdot, B)$ is an $\Pr_{\rv X}$-almost computable function, uniformly in a witness for the $\Pr_{\rv Y}$-almost decidability of $B$. 
	Hence for c.e.\ open sets $E$, the function $\kappa(\cdot, E)$ is lower semicomputable uniformly in a representation for $E$.
Therefore, by Lemma~\ref{phikappa}, the function $\curry \kappa$ is computable, as desired. 
\end{proof}

\begin{remark}
	One might likewise obtain an algorithm 
for conditioning in the context of Proposition~\ref{denstokerncomp}
using the proof of Proposition~\ref{discreteconditioning} via computable rejection sampling. 
In particular, the following classical argument may be carried out computably. For $x\in R$, let
	$\rv E_x$ be a $\Pr_{\rv X}$-almost continuous random variable in $\{0,1\}$ 
such that $\CondProbFuncEval{\rv E_x=1}{\rv Y}{y} =  \frac 1 M p_{\rv X | \rv Y}(x | y)$,
where $M$ satisfies $p_{\rv X | \rv Y}(x | y) < M$ for all $x\in R$ and $y\in T$. (Such an $M$ exists because of the boundedness condition.)
	Then, for every 
	Borel $B \subseteq T$ and every $x\in R$, we have
\[
\kappa (x, B) 
= \frac {\Pr \{ \rv Y \in B, \ \rv E_x = 1 \} }
           {\Pr \{ \rv E_x = 1 \} }
= \Pr \{ \rv Y \in B | \rv E_x = 1 \} ,
\]
providing the reduction to (classical) rejection sampling.
\end{remark}

Because the computability result Proposition~\ref{denstokerncomp} was established in a way that obviously relativizes to any oracle, 
we now obtain a proof of its analogue for a continuous conditional density, as promised in Section~\ref{sec3dom}.

\begin{proof}[Proof of Proposition~\ref{denstokern}]
Follows from the proof of Proposition~\ref{denstokerncomp} by relativizing with respect to an arbitrary oracle.
\end{proof}

\begin{corollary}[Density and independence]\label{independentdensity}
Let $\rv U$, $\rv V$, and $\rv X$ be $\Pr$-almost computable random variables in computable Polish spaces,
where $\rv X$ is independent of $\rv V$ given $\rv U$.  
Assume that there exists a bounded and computable conditional density $p_{\rv X | \rv U}(x|u)$ of $\rv X$ given $\rv U$.
Then the kernel $\CondProbFunc{(\rv U, \rv V)}{\rv X}$ is computable.
\end{corollary}

\begin{proof}
Let $\rv Y = (\rv U,\rv V)$.
Then $p_{\rv X|\rv Y}(x|(u,v)) = p_{\rv X | \rv U}(x|u)$ is the
conditional density of $\rv X$ given $\rv Y$ (with respect to
$\nu$).
The result then follows immediately from Proposition~\ref{denstokerncomp}.
\end{proof}

\subsection{Conditioning on Noisy Observations}

As an immediate consequence of Corollary~\ref{independentdensity},  we obtain the computability of the
following common situation in probabilistic modeling, where the
observed random variable has been corrupted by independent
absolutely continuous noise.

\begin{corollary}[Independent noise]
\label{indnoise}
Let $\rv U$ and $\rv E$ be $\Pr$-almost computable random variables in $\Reals$,
and
let $\rv V$ be a $\Pr$-almost computable random variable in a computable Polish space.
Define $\rv X = \rv U + \rv E$.
If $\Pr_{\rv E}$ is absolutely continuous (with respect to Lebesgue
measure) with a bounded computable density $p_{\rv E}$, and $\rv E$ is independent of $\rv U$ and $\rv V$, then the conditional distribution map
$\CondProbFunc{(\rv U, \rv V)}{\rv X}$
is computable.
\end{corollary}

\begin{proof}
We have that
\[
p_{\rv X|\rv U}(x|u) = p_{\rv E}(x-u)
\]
is the conditional density of $\rv X$ given $\rv U$ (with respect to Lebesgue measure).
The result then follows from Corollary~\ref{independentdensity}.
\end{proof}

Pour-El and Richards \citeyearpar[][Ch.~1, Thm.~2]{MR1005942} show
that a twice continuously differentiable computable function has a
computable derivative (despite the fact that Myhill~\citeyearpar{MR0280373} exhibits a computable function from $[0,1]$ to $\Reals$ whose derivative is continuous, but not computable).
Therefore, noise with a sufficiently smooth computable
distribution has a computable density, and by
Corollary~\ref{indnoise}, if such noise has a bounded density, then an almost computable random variable corrupted by
such noise still admits a computable conditional distribution map.

Furthermore, Corollary~\ref{indnoise} implies that the conditional distribution of a random variable given a continuous observation cannot
always be uniformly approximated using noise in the sense that one cannot computably tell how little noise must be present to obtain a given accuracy.
For example, consider our main construction (Theorem~\ref{actualmainresult})
corrupted with noise $\rv E = \sigma \rv Z$, where $\rv Z$ is a standard Gaussian noise and $\sigma > 0$ determines the standard deviation.
Even though, as 
$\sigma\to 0$, the conditional distribution given the corrupted observation 
converges weakly to the uncorrupted conditional distribution with $\sigma=0$,
the noncomputability of the uncorrupted conditional distribution implies that
one cannot
uniformly compute
a value of $\sigma$ from a desired bound on the error introduced to the conditional distribution corrupted by the noise $\sigma \rv Z$.

\subsection{Exchangeable Setting}
Freer and Roy \citeyearpar{FreerRoyAISTATS2010} show how to compute conditional distributions in the setting of \emph{exchangeable sequences}, i.e., sequences of random variables whose joint distribution is invariant to permutations of the indices.  A classic result by de~Finetti shows that exchangeable sequences of random variables are in fact conditionally i.i.d.\ sequences, conditioned on a random measure, often called the \emph{directing random measure}.   Freer and Roy describe how to transform an algorithm for sampling an \emph{exchangeable sequence} into a rule for computing the
posterior distribution of the directing random measure given observations.
The result is a corollary of a computable version of de Finetti's
theorem \citep{DBLP:conf/cie/FreerR09, MR2880271}, 
and covers a wide range of
common scenarios in nonparametric Bayesian statistics (often where
no conditional density exists).

%%%%%%%%%%%%%%%%%%%%%%%%%%%%%%%%%%%%%%
%%%%%%%%%%%%%%%%%%%%%%%%%%%%%%%%%%%%%%
%%%%%%%%%%%%%%%%%%%%%%%%%%%%%%%%%%%%%%

%%%%%%%%%%%%%%%%%%%%%%%%%%%%%%%%%%%%%%
%%%%%%%%%%%%%%%%%%%%%%%%%%%%%%%%%%%%%%
%%%%%%%%%%%%%%%%%%%%%%%%%%%%%%%%%%%%%%

\begin{acks}
A preliminary version of this article appears as ``Noncomputable conditional distributions'' in
\emph{Proceedings of the 26th Annual IEEE Symposium on Logic in Computer Science},
107--116 \citep{DBLP:conf/lics/AckermanFR11}, which was based upon results first appearing in an arXiv preprint \citep{AFR10}.

C.\ E.\ Freer was partially supported by NSF grants DMS-0901020 and
DMS-0800198, DARPA Contract Award No.\ FA8750-14-C-0001, and
a grant from the John Templeton Foundation.

D.\ M.\ Roy was partially supported
by graduate fellowships from the National Science Foundation and MIT Lincoln Laboratory, a Newton International Fellowship, a Research Fellowship at Emmanuel College, and an Ontario Early Researcher Award.

Work on this publication was also made possible through
the support of ARO grant W911NF-13-1-0212, ONR grant N00014-13-1-0333, 
DARPA Contract Award No.\ FA8750-14-2-0004,
and Google's ``Rethinking AI'' project. 
The opinions expressed in this publication are those of the authors and do not necessarily reflect the views of the John Templeton Foundation or the U.S.\ Government.

The authors would like to thank
Jeremy Avigad,
Henry Cohn,
Leslie Kaelbling,
Vikash Mansinghka,
Arno Pauly,
Hartley Rogers,
Michael Sipser,
and
Joshua Tenenbaum
for helpful discussions,
Bill Thurston for a useful comment regarding
Lemma~\ref{c-infinity-step},
Mart\'in Escard\'o for a helpful discussion regarding computable random variables,
and
Quinn Culver,
Noah Goodman,
Jonathan Huggins,
Leslie Kaelbling,
Bj{\o}rn Kjos-Hanssen,
Vikash Mansinghka,
Timothy O'Donnell,
Geoff Patterson,
and the anonymous referees
for comments on earlier drafts.
\end{acks}
%%%%%%%%%%%%%%%%%%%%%%%%%%%%%%%

%\newpage
%\bibliographystyle{acmsmall}
\bibliographystyle{ACM-Reference-Format}

%\bibliography{conditioning}


\def\cprime{$'$} \def\cprime{$'$}
  \def\polhk#1{\setbox0=\hbox{#1}{\ooalign{\hidewidth
  \lower1.5ex\hbox{`}\hidewidth\crcr\unhbox0}}}
  \def\polhk#1{\setbox0=\hbox{#1}{\ooalign{\hidewidth
  \lower1.5ex\hbox{`}\hidewidth\crcr\unhbox0}}} \def\cprime{$'$}
  \def\cprime{$'$} \def\cprime{$'$} \def\cprime{$'$}

%%%%%%%%%%
\end{document}